\documentclass[reqno]{amsart}
\usepackage{hyperref}
\usepackage{amsfonts}
\usepackage{amsthm,amsmath}
\usepackage[integrals]{wasysym}
\usepackage{enumerate}
\usepackage{fancyhdr}
\usepackage{amssymb}
\usepackage{chemarrow}
\usepackage{tikz}
\usepackage{mathtools}
\usepackage{mathrsfs}

\usepackage{pdftexcmds}
\numberwithin{equation}{section}
\newtheorem{theorem}{Theorem}[section]
\newtheorem{lemma}[theorem]{Lemma}
\newtheorem{definition}[theorem]{Definition}
\newtheorem{proposition}[theorem]{Proposition}
\theoremstyle{remark}
\newtheorem{remark}{Remark}
\allowdisplaybreaks

\newcommand{\andf}{\quad\hbox{and}\quad}

 \newcommand{\wt}{\widetilde}

\def\div{ \hbox{\rm div}\,  }
\def\curl{ \hbox{\rm curl}\,  }

\def\R{{\mathbb R}}

\def\Z{{\mathbb Z}}

\def\ddj{\dot\Delta_j}
\def\eps{\varepsilon}

\begin{document}
\title[\hfilneg \hfil ]
{Global Well-posedness for the  Hall-magnetohydrodynamics system  in larger  critical Besov spaces}

 \author[Liu]{ Lvqiao Liu}
 \address[Lvqiao Liu]{\newline School of Mathematics and statistics, Wuhan University, Wuhan 430072, China}
 \email{lvqiaoliu@whu.edu.cn}
 
  \author[Tan]{ \text{Jin Tan}}
 \address[Jin Tan]{\newline Universit\'e Paris-Est Cr\'eteil,  LAMA UMR 8050, 61 avenue du G\'en\'eral de Gaulle,  94010 Cr\'eteil cedex 1}
 \email{jin.tan@u-pec.fr}
 

\begin{abstract}
 We prove the global well-posedness of  the Cauchy problem to the 3D incompressible  Hall-magnetohydrodynamic system 
supplemented  with initial data in critical Besov spaces, which generalize the result in \cite{2019arXiv191103246D}.
Meanwhile, we analyze the long-time behavior of the solutions and get some decay estimates.
Finally, a stability theorem for global  solutions is established. \\

\end{abstract}

\maketitle
\section{Introduction }
This paper focuses on  the following three dimensional incompressible resistive and viscous Hall-magnetohydrodynamics system (Hall-MHD) in $\mathbb{R}^3$:
\begin{align}
&\partial_t{\mathnormal u}+u\cdot\nabla u+\nabla \mathnormal \pi= b \cdot \nabla  b+\mu\Delta \mathnormal u\label{1.1},\\
&{\mathrm{div}}\,\mathnormal u={\mathrm{div}}\,\mathnormal   b=0\label{1.2},\\
&\partial_t{\mathnormal b}-\nabla\times((\mathnormal u -{\varepsilon}\nabla\times\mathnormal b) \times \mathnormal b)=\nu\Delta \mathnormal b\label{1.3}, 
\end{align}
 with the initial data:
\begin{equation}
(\mathnormal u(0,\mathnormal x), \mathnormal b(0,\mathnormal x))=(\mathnormal u_{0}(\mathnormal x), \mathnormal b_{0}(\mathnormal x)),\quad
x\in{\mathbb R}^3.\label{1.4}
\end{equation}
Where $\pi=(P+\frac{|b|^2}{2})$,    $\mathnormal u$, $\mathnormal b$ and $\mathnormal P$ stand for  the velocity field, the magnetic field and the scalar pressure, respectively.  The parameters $\mu$ and $\nu$ denote the fluid viscosity and the magnetic resistivity respectively, while the dimensionless number ${\varepsilon}$ measures   the 
magnitude  of the Hall  effect compared to  the typical length scale of the fluid.

 Hall-MHD is much different from the classical MHD equation, due to the appearance of the so-called {{\em{Hall-term}} $:\eps\nabla\times((\nabla\times b)\times b)$.  It does play an important role in magnetic reconnection, as observed in e.g. 
 plasmas, star formation, solar flares, neutron stars or geo-dynamo. For  more explanation on the physical background of Hall-MHD system, one can refer to \cite{BalbusLinear,Braiding2005Star,MR1994506,SomovMagnetic}.
Meanwhile, it looks that the mathematical analysis of the Hall-MHD system is  more complicated than that for the MHD system, since Hall-term makes Hall-MHD  a quasi-linear PDEs.

Considering its physical significance and mathematical applications,  Hall-MHD system has been considered 
by many researchers. The authors in \cite{MR2861579} had derived the Hall-MHD equations from a two-fluid Euler-Maxwell system for electrons and ions by some scaling limit arguments, which also provided a kinetic formulation for the Hall-MHD.
  Then, in \cite{MR3208454}, Chae, Degond and  Liu showed the global existence of weak solutions as well as the local well-posedness of classical solutions with initial data in sobolev spaces $H^{s}$ with $s>5/2$.  Weak solutions have been further investigated 
by  Dumas and Sueur in \cite{MR3227510}.  
  Moreover, Blow-up criteria for smooth solutions and the small data global existence of smooth solutions are obtained in \cite{MR3186849, Wz15}. Later,  in \cite{MR3460222,MR3460237}, Weng studied the long-time behaviour and obtained optimal space-time decay rates of strong solutions.   More recently, \cite{MR3513590, Wz19}  established the well-posedness of strong solutions with improved regularity conditions for initial data in sobolev  or Besov spaces, and  smooth data with arbitrarily large  $L^\infty$ norms
giving rise to  global unique solutions have been exhibited in \cite{2019arXiv190302299L}.  Very recently,  Danchin and the second author in \cite{2019arXiv191103246D,2019arXiv191209194D}  establish well-posedness  in critical spaces based on a new observation of the Hall-MHD system.
\smallbreak
Our first goal  here is to prove  the global well-posedness  of Hall-MHD system  with initial data in \emph{larger critical spaces} compared with \cite{2019arXiv191103246D}.   Let us first  recall the classical MHD system (corresponding to ${\varepsilon}=0$):
$$\left\{
\begin{aligned}
&\partial_t{\mathnormal u}+u\cdot\nabla u+\nabla \mathnormal \pi= b\cdot \nabla  b+\mu\Delta \mathnormal u,\\
&\div  u=\div   b=0,\\
&\partial_t{\mathnormal b}-\nabla\times(\mathnormal u \times\mathnormal b) =\nu\Delta \mathnormal b, 
\end{aligned}
\right.\leqno(MHD)$$
it is  invariant for all  $\lambda>0$ by the rescaling 
\begin{equation}\label{eq:solutionMHD}
(u(t,x), b(t, x))\leadsto  \lambda (u({\lambda^{2}t, \lambda x}), b({\lambda^{2}t, \lambda x}))\andf \pi(t,x)\leadsto  
\lambda^2\pi({\lambda^{2}t, \lambda x}),
\end{equation}
provided the initial  {data}  $(u_0, b_0)$ is rescaled according to 
\begin{equation}\label{eq:dataMHD}
(u_0(x),  b_0(x))\leadsto(\lambda u_0(\lambda x),  \lambda b_0(\lambda x)).
\end{equation}
One can refer \cite{MR2477101} for the well-posedness of $(MHD)$ in critical Besov spaces.

But we see the Hall-term in \eqref{1.3} breaks the above scaling and system \eqref{1.1}-\eqref{1.3} does not have any $scaling~invariance$. It is pointed out in \cite{MR3186849} that if we set the fluid velocity $u$ to $0$ in \eqref{1.3}, 
then we get the following \emph{Hall equation} for $b:$
$$\left\{\begin{aligned}
&\partial_t{\mathnormal b}+{\varepsilon}\nabla\times((\nabla\times\mathnormal b)\times \mathnormal b)
=\nu\Delta \mathnormal b,\\
&b|_{t=0}=b_0,\end{aligned}\right.\leqno(Hall)$$
which is invariant by the rescaling 
\begin{equation}\label{eq:solutionHall}
b(t,x)\leadsto  b({\lambda^{2}t, \lambda x}),
\end{equation}
provided the data $b_0$ is rescaled according to 
\begin{equation}\label{eq:dataHall}
b_0(x)\leadsto b_0(\lambda x).
\end{equation}
In other words,  $\nabla b$  has  the same scaling invariance as the fluid velocity $u$ 
in $(MHD)$.  

On the another hand, Danchin and the second author \cite{2019arXiv191103246D} have transformed the Hall-MHD system into a system having some $scaling ~invariance$, they consider  the current function  $J:=\nabla\times b$ as an additional unknown. 
 Since $b$ is divergence free,  then thanks to the vector identity
 $$\nabla\times (\nabla\times v)+\Delta v=\nabla\div v,$$
we have
$$
 b={\rm{curl}}^{-1}J\,{:=}(-\Delta)^{-1}\nabla\times J,$$
where the $-1$-th order homogeneous Fourier multiplier ${\rm{curl}}^{-1}$ is defined 
 on the Fourier side by 
\begin{equation}\label{eq:curl-1}\mathcal{F}({\rm{curl}}^{-1}J)(\xi):=\frac{i\xi\times\widehat{J}}{|\xi|^2}\cdotp\end{equation}

With that notation,  the system \eqref{1.1}-\eqref{1.3} can be extended to the following \emph{extended Hall-MHD system}:
\begin{equation}\label{main1}
\left\{\begin{aligned}
&\partial_t{\mathnormal u}+u\cdot\nabla u-\mu\Delta \mathnormal u+\nabla \mathnormal \pi=b\cdot \nabla  b,\\
&\partial_t{\mathnormal b}-b\cdot\nabla(u-\eps J)+(\mathnormal u - \eps J)\cdot\nabla b-\nu\Delta \mathnormal b=0,\\
 &\partial_t{\mathnormal J}-\nabla\times\bigl(\nabla\times((\mathnormal u - \eps J) \times {\rm{curl}}^{-1}\mathnormal J)\bigr)-\nu\Delta \mathnormal J=0,\\
 &{\mathrm{div}}\,\mathnormal u={\mathrm{div}}\,\mathnormal b=\div J=0.
\end{aligned}\right.
\end{equation}
with initial data
\begin{equation}\label{maini}
(\mathnormal u(0,\mathnormal x), \mathnormal b(0,\mathnormal x),  \mathnormal J(0,\mathnormal x))=(\mathnormal u_{0}(\mathnormal x), \mathnormal b_{0}(\mathnormal x), \mathnormal \nabla\times b_{0}(\mathnormal x)),\quad
x\in{\mathbb R}^3.
\end{equation}
The reason of considering the  extended system \eqref{main1} rather than the initial system \eqref{1.1}-\eqref{1.3} is
that it has a $scaling ~invariance$, which is actually the same as that of the incompressible MHD equations.
Even better, the quadratic terms in the first two lines of \eqref{main1} are keep the same type with
the incompressible MHD equations.  It is thus natural to study whether the above system goes beyond the theory of the generalized Navier-stokes equations as presented in e.g. \cite{MR2768550}.
 
 {Next, we focus on the large time behavior of the solution in critical Besov spaces
 by using time-weighted estimates.}
 
Our third purpose is to prove the stability of an $a~priori$ global solution to Hall-MHD system. Compared with the classical incompressible Navier-Stokes equations which is semi-linear, the Hall-MHD system is quasi-linear, it may forces us to go beyond
the theory of the generalized Navier-Stokes equations, since the differentiation is outside instead of being inside on the $\curl^{-1} J$ in the last line of \eqref{main1}.
However,  in the case $\mu = \nu,$ it is possible to take advantage of the cancellation property found in \cite{2019arXiv191103246D} combined with  standard energy method to recover  the 
stability results as Navier-Stokes equation (see \cite{MR2032938}). To this, we have to take $p=q=2.$

Throughout this paper, we use $C$ to denote a general positive constant which may
different from line to line. And we sometimes write $A\lesssim B$ instead of $A\leq C B.$  Likewise,    $A\sim B$ means that  $C_1 B\leq A\leq C_2 B$ with absolute constants $C_1$, $C_2$. For $X$ a Banach space, $p\in[1, \infty]$ and $T>0$, the notation $L^p(0, T; X)$ or $L^p_T(X)$ designates the set of measurable functions $f: [0, T]\to X$ with $t\mapsto\|f(t)\|_X$ in $L^p(0, T)$, endowed with the norm $\|\cdot\|_{L^p_{T}(X)} :=\|\|\cdot\|_X\|_{L^p(0, T)},$ and agree that $\mathcal C([0, T], X)$ denotes the set of continuous functions from $[0, T]$ to $X$. Slightly abusively, we will keep the same notations for multi-component functions.


\section{Main results}

After the work of  \cite{MR1145160}, we know that the incompressible Navier-Stokes equations is locally well-posed in all homogeneous 
Besov spaces $\dot B^{\frac3p-1}_{p,r}$ with $1\leq p<\infty$ and $1\leq r\leq \infty,$ for any initial data, and globally well-posed for small initial data (see \cite{MR2477101} for similar results of MHD equations).  
Once the Hall-MHD system \eqref{1.1}-\eqref{1.3}  has been recast into its
 extended version as in \eqref{main1}, 
 compared to a recent work obtained  in \cite{2019arXiv191103246D}, they prove global well-posdedness for initial data $$(u_0, b_0, J_0)\in \dot{B}_{p,1}^{\frac{3}{p}-1}\times\dot{B}_{p,1}^{\frac{3}{p}-1}\times\dot{B}_{p,1}^{\frac{3}{p}-1}\quad{\rm{with}}\quad1\leq p<\infty.$$
 It is interesting to consider well-posedness for initial data   $(u_0, b_0, J_0)$
 in the general critical homogeneous Besov spaces  $$\dot{B}_{p,1}^{\frac{3}{p}-1}\times\dot{B}_{q,1}^{\frac{3}{q}-1}\times\dot{B}_{q,1}^{\frac{3}{q}-1}.$$
 The solution $(u, b)$ thus lies in the {space} $E_p\times E_q$, which defines as follow:
 $$E_{p}(T){:=} \{v\in \mathcal{C}([0, T], \dot B^{\frac{3}{p}-1}_{p, 1}),~\nabla_x^2 v\in  L^1(0, T; \dot B^{\frac{3}{p}-1}_{p, 1})\,~{\rm{and}}~\div~v=0 \}
 $$
 with
 \begin{equation*}
 \|v\|_{E_p(T)}:=\|v\|_{L^\infty_T(\dot B^{\frac{3}{p}-1}_{p, 1})}+\|v\|_{L^1_T(\dot B^{\frac{3}{p}+1}_{p, 1})}.
 \end{equation*}
 or in its global version, denoted by  $E_{p},$ for solutions defined on $ {\mathbb{R}_+}\times \mathbb{R}^3.$ 

\medbreak
Our first result states the global well-posedness of the  Hall-MHD system \eqref{1.1}-\eqref{1.4} for all positive coefficients $\mu, \nu, \eps.$ 

\begin{theorem}\label{th}
	Let $1 \leq p\leq q <\infty$ be such that
	\begin{equation}\label{condition1}
	-\min\{\frac13,~\frac1{2p}\}\leq \frac{1}{q}-\frac{1}{p}.
	\end{equation}
	Assume that $u_0 \in \dot B^{\frac{3}{p}-1}_{p, 1} $ and $  b_0,\,\nabla\times b_0 \in  \dot B^{\frac{3}{q}-1}_{q, 1}.$
There exists a positive constant $\varepsilon_0$ depends on $\mu, \nu, \eps, p, q$ such that if
\begin{equation}\label{small}
	\|u_0\|_{\dot B^{\frac{3}{p}-1}_{p, 1}}+\|b_0\|_{\dot B^{\frac{3}{q}-1}_{q, 1}}+\|\nabla\times b_0\|_{\dot B^{\frac{3}{q}-1}_{q, 1}}\leq \varepsilon_0,
	\end{equation}  
	then the Cauchy problem \eqref{1.1}-\eqref{1.4} admits a unique global-in-time solution 
	\begin{equation}\label{562875}
	(u, b ) \in E_p\times E_q , \quad \text{and} \quad \nabla\times b\in E_{q},
	\end{equation}
	with
	\begin{equation}\label{1.1200}
	\begin{split}
	&\|u\|_{L^\infty(\mathbb{R}_+; \dot{B}_{p, 1}^{\frac{3}{p}-1})}+\mu\|u\|_{L^1(\mathbb{R}_+; \dot{B}_{p, 1}^{\frac{3}{p}+1})}\\
	&\hspace{3cm}+\|(b,  J)\|_{L^\infty(\mathbb{R}_+; \dot{B}_{q, 1}^{\frac{3}{q}-1})}+\nu\|(b,  J)\|_{L^1(\mathbb{R}_+; \dot{B}_{q, 1}^{\frac{3}{q}+1})}\leq2\varepsilon_0 .
		\end{split}
	\end{equation}
	If only $\nabla\times b_0$ fulfills \eqref{small} and in addition
	\begin{equation}\label{condition2}
	-\frac13<\frac1q-\frac1p,
	\end{equation}
	there exists a time $T>0$ such that Hall-MHD system admits a unique local-in-time solution $$(u, b)\in E_p(T)\times E_q(T)\quad{\text{with}}\quad \nabla\times b\in E_q(T).$$
\end{theorem}	
 
  Next,  we prove that  the solution  has the following decay estimates.
 \begin{theorem}\label{th1}
 Let $1\leq p\leq q<\infty$ satisfy \eqref{condition1} and \eqref{condition2}. 
Let $(u, b)$ be a solution of the Cauchy problem \eqref{1.1}-\eqref{1.4}  supplemented with initial data $(u_0, b_0)$ that satisfies \eqref{small}. Then for any integer $m\geq 1$, we have 
 	\begin{equation*}
 	\|D^m u\|_{\dot{B}_{p, 1}^{\frac{3}{p}-1}}+\|D^m b\|_{\dot{B}_{q, 1}^{\frac{3}{q}-1}} \leq C_0 \eps_0  t^{ -\frac m2},
 	\end{equation*}
 	for all $t>0,$  where 
 	$$\|D^m u\|_{\dot{B}_{p, 1}^{\frac{3}{p}-1}}:=\sup_{|\alpha|=m}\|D^\alpha u\|_{\dot{B}_{p, 1}^{\frac{3}{p}-1}},$$ and the positive constant $C_0$ depends on $\mu, \nu,\eps, p, q, m.$
 \end{theorem}

 The following Theorem states the global stability for possible large solutions of Hall-MHD system in critical spaces when $\mu=\nu.$
 \begin{theorem}\label{th2}
     Assume that  $(u_{0,i}, b_{0,i})\in{\dot B^{\frac{1}{2}}_{2, 1}}(\R^3)$ with $\div u_{0,i}=\div b_{0,i}=0$
such that  $v_{0,i}\in{\dot B^{\frac{1}{2}}_{2, 1}}(\R^3),$ 
where $$v_{0,i}:=u_{0,i}-\eps \nabla\times b_{0,i}, ~i=1, 2.$$
Suppose in addition that for $\mu=\nu$  the Cauchy problem \eqref{1.1}-\eqref{1.4}
supplemented with initial data $(u_{0,1},b_{0,1})$ admits a global solution  $(u_1,b_1)$ 
 such that
 $$(u_1, b_1,\nabla\times b_1)\in L^1(\R_+; \dot{B}^\frac52_{2, 1}(\R^3)).$$ 
     There exist two positive constants   $\eta, C$  {depend}  on $\mu,$ $\eps$  such that if
     \begin{equation}\label{ind}
         \|(u_{0,1}-u_{0,2},b_{0,1}-b_{0,2}, v_{0,1}-v_{0,2} )\|_{\dot B^{\frac{1}{2}}_{2, 1}}\leq\eta,
     \end{equation}
     then $(u_{0,2}, b_{0,2})$ generate a global solution $(u_2, b_2)\in E_2$, and
     \begin{equation}\label{solud}
         \|(u_{1}-u_{2},b_{1}-b_{2}, \nabla\times b_{1}-\nabla\times b_{2})\|_{E_2}\leq \eta\exp\bigl(C\|(u_1, b_1, \nabla\times b_1)\|_{E_2}\bigr).
     \end{equation}
 \end{theorem}
 
 \begin{remark}
 	 In this Theorem, we prove that the flow associated to the Hall-MHD system is Lipschitz in critical regularity setting: perturbing a global solution gives again a global solution, which moreover stays close to the given one. It improve the result in \cite{MR3513590}.
 \end{remark}
 \begin{remark}
 As proposed by Chae and Lee in \cite{MR3186849}, considering the $2\frac12$D flows for the Hall-MHD system, which reads:
\begin{align*}
&\partial_t{\mathnormal u}+\wt u\cdot\wt\nabla u+\wt\nabla\pi=\wt b\cdot\wt\nabla b+\mu\wt\Delta u&\hbox{in }\ \R_+\times\R^2,\\
&\widetilde{\div}\wt u=\wt\div b=0&\hbox{in }\ \R_+\times\R^2,\\
&\partial_t{\mathnormal b}-\wt\nabla\times((u-\eps j)\times b)=\nu\wt\Delta b&\hbox{in }\ \R_+\times\R^2, 
\\
&(\mathnormal u, \mathnormal B)|_{t=0}=(\mathnormal u_{0},\mathnormal B_{0})
&\hbox{in}\ \R^2,
\end{align*}
where the unknowns $u$ and $b$ are functions from $\R_+\times\R^2$ to $\R^3,$
 $\wt u:=(u^1, u^2),$ $\wt{b}:=(b^1, b^2),$ $\wt\nabla:=(\partial_1, \partial_2),$ 
$\wt\div :=\wt\nabla\cdot,$
$\wt\Delta:=\partial_1^2+\partial_2^2$ and 
 $$   j:=\wt\nabla\times b= \left(  \begin{array}{c}
          \partial_2 b^{3} \\
         -\partial_1b^{3}\\
         \partial_1 b^{2}-\partial_2 b^{1}
          \end{array}
\right)\cdotp$$
After a small diversification,  our method may still works for this case. And we shall see that for any initial data $(u_0, 0)\in\dot{B}^0_{2, 1}(\R^2),$ it will generate a global solution $(u, 0)$ for the above system due to the theory of  $2\frac12$D Navier-Stokes equation in \cite{BM}, thus one can conclude that if  $$\|(b_0, \nabla\times b_0)\|_{\dot{B}^0_{2, 1}}<\eta$$ 
then supplemented with data $(u_0, b_0)$ will also generate a global solution.
 \end{remark}
\section{Preliminaries}

In this section, we first recall the Littlewood-Paley decomposition theory, the definition of homogeneous Besov space and some useful properties.   More details and proofs may be found in  e.g.  \cite{MR2768550}. 

Let  $ {\varphi}\in\mathscr{D}(\mathscr{C})$ be a smooth function supported in the annulus   $\mathcal{C}=\{k\in\mathbb{R}^3:\frac34\leq |\xi|\leq \frac83\}$ and such that 
\begin{gather*}
\sum_{j\in\mathbb{Z}}  {\varphi}(2^{-j}k)=1,\quad \forall \xi\in\mathbb{R}^3    \backslash \{0\}.
\end{gather*}
For $u\in \mathcal{S}'(\mathbb{R}^3)$, the frequency localization operator  $\dot{\Delta}_j$ and $\dot{S}_j$ are defined by 
\begin{equation*}
 \forall j \in \mathbb{Z},\quad \dot{\Delta}_ju:={\varphi}(2^{-j} D)u ~\text{and}~\dot{S}_j u:=\sum_{\ell \leq j-1 }\dot{\Delta}_{\ell} u.
\end{equation*}
Then we have the formal decomposition
\begin{equation*}
u=\sum_{j\in\mathbb{Z}}\dot{\triangle}_ju, \quad \forall\, u \in{\mathcal S}'_h({\mathbb R}^3):= \mathcal{S}'(\mathbb{R}^3)/{\mathscr{P}}[\mathbb{R}^3].
\end{equation*}
where $\mathscr{P}[\mathbb{R}^3]$ is the set of polynomials (see \cite{MR0461123}). Moreover, the Littlewood-Paley decomposition satisfies the property of almost orthogonality:
\begin{equation*}
{\dot\Delta}_j{\dot\Delta}_k u=0, \quad {\rm{if}} \ |j-k|\geq 2, \quad \dot\Delta_j(S_{k-1}u \dot\Delta_k u)=0, \quad {\rm{if}} \ |j-k|\geq 5.
\end{equation*}
We now recall the definition of homogeneous Besov spaces from \cite{MR2768550}.
\begin{definition}
	Let $s$ be a real number and $(p, r)$ be in $[1,\infty]^{2}$, we set
	\begin{equation*}
	\|u\|_{\dot B^{s}_{p, r}} :=  \left\{
	\begin{split}
	&\|2^{js}\|\dot\Delta_{j}u\|_{L^{p}(\mathbb{R}^d)}\|_{\ell^r(\mathbb{Z})} \quad\, {\rm{for}} ~ 1 \leq r < \infty,\\
	& \sup_{j \in \mathbb{Z}} 2^{js}\|\dot\Delta_{j}u\|_{L^{p}}~\,\,\quad\qquad {\rm{for}} ~  r= \infty.
	\end{split}
	\right.
	\end{equation*}
	The  homogeneous Besov space $\dot B^{s}_{p, r}  :=\{ u \in \mathcal{S}^{'}_{h}(\mathbb{R}^3),\,	\|u\|_{\dot B^{s}_{p, r}}<\infty \}$.
	
\end{definition}
Next, we recall some basic facts on Littlewood-Paley theory and Besov spaces, one
may check \cite{MR2768550,MR631751} for more details.
\begin{proposition}\label{563856} 
	Fix some $0<r<R.$ 
	A constant $C$ exists such that for any nonnegative integer $k$, any couple $(p, q)$ in $[1, \infty]^2$ with $q\geq p\geq 1$ and any function $u$ of $L^p$ with 
	${\rm{Supp}}~\widehat u\subset \{\xi\in{\mathbb R}^d,\; |\xi|\leq \lambda R\},$ we have
	\begin{equation*}
	 \|D^{k}u\|_{L^{q}}\leq C^{k+1}\lambda^{k+d(\frac{1}{p}-\frac{1}{q})}\|u\|_{L^{p}}.
	\end{equation*}
	If $u$ satisfies   ${\rm{Supp}}~\widehat u\subset \{\xi\in{\mathbb R}^d,\; r\lambda\leq |\xi|\leq R\lambda\},$  then we  have
	\begin{equation*}
	C^{-k-1}\lambda^{k}\|u\|_{L^{p}}\leq   \|D^{k}u\|_{L^{p}} \leq C^{k+1}\lambda^{k}\|u\|_{L^{p}}.
	\end{equation*}
\end{proposition}
	\begin{lemma}\label{semi}
	Let $\mathcal{C}$ be a ring of $\mathbb{R}^3$, if the support of $\hat{u}$ is included  in $\lambda\mathcal{C}$. Then, there exist two positive constants  $c$ and $ C$ such that for all $1\leq p \leq \infty$,
	\begin{equation*}
	\|e^{t\Delta}u\|_{L^p} \leq C e^{c\lambda^2 t}\| u\|_{L^p},
	\end{equation*} 
	where $e^{t\Delta}$ denotes the heat semi-group operator.
\end{lemma}
\begin{proposition}\label{P_Besov}
	Let $1\leq p \leq \infty$. Then there hold:
	
	\begin{itemize}
		\item for all $s\in\mathbb R$~~and~~$1\leq p, r\leq\infty,$ 
		we have 
			\begin{equation*}
		\|D^{k}u\|_{\dot B^{s}_{p, r}}\sim\|u\|_{\dot B^{s+k}_{p, r}}.
		\end{equation*}
		\item for any $\theta\in(0, 1)$ and $\,s<\tilde s,$ we have
$$\|u\|_{\dot B^{\theta s+(1-\theta)\tilde{s}}_{p, 1}}\lesssim\|u\|_{\dot B^{s}_{p, 1}}^{\theta}\|u\|_{\dot B^{\tilde{s}}_{p, 1}}^{1-\theta}.$$
	\end{itemize}
\end{proposition}
\begin{lemma}\label{FM}
Let $f$ be a smooth function on $\mathbb{R}^{3}\setminus\{0\}$ which is homogeneous of degree m.  Let $1\leq p, r\leq \infty.$ Assume that
$$s-m<\frac{3}{p}, \quad{\rm{or}}\quad s-m=\frac{3}{p} \quad{\rm{and}}\quad  r=1.$$
Define $f(D)$
			on ${\mathcal S}'_h({\mathbb R}^3)$ by
			\begin{equation*}
			\mathcal{F}(f(D)u)(\xi){:=}f(\xi)\mathcal{F}u(\xi),
			\end{equation*}
			and assume that $f(D)$ maps ${\mathcal S}'_h({\mathbb R}^3)$ to itself. 
			Then $f(D)$ is continuous from $\dot B^{s}_{p, r}({\mathbb R}^3)$ to $\dot B^{s-m}_{p, r}({\mathbb R}^3)$.
\end{lemma}
In the next, we shall need to use Bony's decomposition from \cite{MR631751} in the homogeneous context:
$$u\,v=T_uv+T_vu+R(u, v)$$
with 
$$T_uv := \sum_{q}\dot{S}_{q-1}u{\dot\Delta_q} v,\quad {\rm{and}}\quad R(u, v) :=\sum_{q}\sum_{|q'-q|\leq 1}{\dot\Delta_q} u\dot{\Delta}_{q'}v.$$
The above operator $T$ is called the "paraproduct" whereas $R$ is called the "remainder".  
\begin{lemma}\label{P_para}
	Let $(s, p, r)\in\mathbb{R}\times[1, \infty]^2$ and $t<0$, there exists a constant $C$ such that
	$$\|T_uv\|_{\dot{B}^{s}_{p, r}}\leq C\,\|u\|_{L^\infty}\|v\|_{\dot{B}^s_{p, r}}\quad{\rm{and}}\quad \|T_uv\|_{\dot{B}^{s+t}_{p, r}}\leq C\,\|u\|_{\dot{B}^{t}_{\infty, \infty}}\|v\|_{\dot{B}^s_{p, r}}\cdotp$$
	For any $(s_1, p_1, r_1)$ and $(s_2, p_2, r_2)$ in $\mathbb{R}\times[1, \infty]^2$ there exists a constant $C$ such that  if $s_1+s_2>0$, $\frac{1}{p} \leq\frac{1}{p_1}+\frac{1}{p_2}\leq 1$ and \,$\frac{1}{r} \leq\frac{1}{r_1}+\frac{1}{r_2}\leq 1$ then
	$$\|R(u, v)\|_{\dot{B}^{s}_{p, r}}\leq C\,\|u\|_{\dot{B}^{s_1}_{p_1, r_1}}\|v\|_{\dot{B}^{s_2}_{p_2, r_2}},$$
	with $s:=s_1+s_2-3(\frac{1}{p_1}+\frac{1}{p_2}-\frac{1}{p}),$ provided that $s< \frac{3}{p}$ or $s= \frac{3}{p}$ and $r=1.$
\end{lemma}
As an application of the above basic facts on Littlewood-Paley theory, the following product laws in Besov spaces will play a crucial role in the sequel . 

\begin{lemma}\label{law1}(see \cite{MR2889168})
	Let $q\geq p\geq 1,$ and $s_1\leq \frac{3}{p},$ $s_2\leq \frac{3}{q}$ with
	\begin{equation*}\label{562753}
	s_1+s_2\geq 3\min\{ 0, \frac{1}{p}+\frac{1}{q}-1\}.
	\end{equation*}
	Let $a\in \dot{B}^{s_1}_{p, 1}(\R^3),$ $b\in \dot{B}^{s_2}_{q, 1}(\R^3)$. Then $a \,b \in \dot B^{s_1+s_2-\frac{3}{p}}_{q, 1}(\R^3),$ and 
	\begin{equation}\label{law_1}
	\|a\,b\|_{\dot B^{s_1+s_2-\frac{3}{p}}_{q, 1}}\lesssim \|a\|_{\dot B^{s_1}_{p, 1}} \|b\|_{\dot B^{s_2}_{q, 1}}\cdotp
	\end{equation}
\end{lemma}

\begin{lemma}\label{law0}
	Let $q\geq p\geq 1$ and
	\begin{equation*}\label{562752}
	\frac{1}{q}-\frac{1}{p}\geq -\min\{ \frac{1}{3}, \frac{1}{2p}\}.
	\end{equation*}
	Assume $\theta$ satisfies
	$$\frac{3}{p}-\frac{3}{q}\leq\theta\leq 1.$$
	Let $a, b\in \dot{B}^{\frac{3}{q}-\theta}_{q, 1}(\R^3)\cap\dot{B}^{\frac{3}{q}+\theta}_{q, 1}(\R^3)$. Then $a\,b \in \dot B^{ \frac{3}{p }}_{p, 1}(\R^3),$ and 
	\begin{equation}\label{law_0}
	\|a\,b\|_{\dot B^{\frac{3}{p}}_{p, 1}}\lesssim \|a\|_{\dot B^{\frac{3}{q}-\theta}_{q, 1}} \|b\|_{\dot B^{\frac{3}{q}+\theta}_{q, 1}}+\|a\|_{\dot B^{\frac{3}{q}+\theta}_{q, 1}} \|b\|_{\dot B^{\frac{3}{q}-\theta}_{q, 1}}\cdotp
	\end{equation}
Let $a, b\in \dot{B}^{\frac{6}{q}-\frac3p}_{q, 1}(\R^3)\cap\dot{B}^{\frac{3}{q}}_{q, 1}(\R^3)$. Then $a\,b \in \dot B^{ \frac{3}{q }}_{p, 1}(\R^3),$ and 
\begin{equation}\label{law_000}
	\|a\,b\|_{\dot B^{\frac{3}{q}}_{p, 1}}\lesssim \|a\|_{\dot B^{\frac{3}{q}}_{q, 1}} \|b\|_{\dot B^{\frac{6}{q}-\frac3p}_{q, 1}}+\|a\|_{\dot B^{\frac{6}{q}-\frac3p}_{q, 1}} \|b\|_{\dot B^{\frac{3}{q}}_{q, 1}}\cdotp
	\end{equation}
\end{lemma}
\begin{proof}
The proof is standard, we follow the method of \cite{MR2889168} and only give a proof of \eqref{law_0}. In fact, we shall focus on the case $q>p,$ since when $q=p,$ it is obvious.

\noindent By Bony's decomposition, we can write
$$ab=T_a b+T_b a +R(a, b).$$
By H\"older's inequality and Proposition \ref{563856}, under $p<q\leq 2p$ we have 
\begin{align*}
2^{\frac{3}{p}j}\|S_{j-1} a \dot\Delta_j b\|_{L^p}&\lesssim (2^{(\frac{3}{q}+\theta)j}\|\dot\Delta_j b\|_{L^q})\sum_{k\leq j-2}2^{(\frac{3}{p}-\frac{3}{q}-\theta)k}\|\dot\Delta_k a\|_{L^{\frac{pq}{q-p}}}2^{(\frac{3}{p}-\frac{3}{q}-\theta)(j-k)}\\
&\lesssim (2^{(\frac{3}{q}+\theta)j}\|\dot\Delta_j b\|_{L^q})\sum_{k\leq j-2}2^{(\frac{3}{q}-\theta)k}\|\dot\Delta_k a\|_{L^{q}}2^{(\frac{3}{p}-\frac{3}{q}-\theta)(j-k)}.
\end{align*}
Thus by Young's inequality with $\frac{3}{p}-\frac{3}{q}\leq\theta$
\begin{align*}
\|T_a b\|_{\dot{B}^{\frac{3}{p}}_{p, 1}}\lesssim \|a\|_{\dot B^{\frac{3}{q}-\theta}_{q, 1}} \|b\|_{\dot B^{\frac{3}{q}+\theta}_{q, 1}}\cdotp
\end{align*}
Similarly,
\begin{align*}
\|T_b a\|_{\dot{B}^{\frac{3}{p}}_{p, 1}}\lesssim \|a\|_{\dot B^{\frac{3}{q}+\theta}_{q, 1}} \|b\|_{\dot B^{\frac{3}{q}-\theta}_{q, 1}}\cdotp
\end{align*}
Thanks to Lemma \ref{P_para}, we have
\begin{align*}
\|R(a, b)\|_{\dot{B}^{\frac{3}{p}}_{p, 1}}&\lesssim \|a\|_{\dot B^{\frac{3}{p}-\frac3q-\theta}_{\frac{pq}{q-p}, \infty}} \|b\|_{\dot B^{\frac{3}{q}+\theta}_{q, 1}}\\
&\lesssim \|a\|_{\dot B^{\frac{3}{q}-\theta}_{q, 1}} \|b\|_{\dot B^{\frac{3}{q}+\theta}_{q, 1}}.
\end{align*}
\end{proof}

\begin{lemma}\label{law2}
	Let $1\leq q < \infty.$ For any homogeneous function $\sigma$ of degree - 1 smooth outside of 0, there hold:
	\item-- let $a \in \dot B^{\frac{3}{q}-1}_{q, 1}$ and $b \in \dot B^{\frac{3}{q}+1}_{q, 1},$ then 
	\begin{equation}\label{5682396529}
	\|(\sigma(D)a )\cdot\nabla b\|_{\dot B^{\frac{3}{q} }_{q, 1} }\lesssim \|a \|_{ \dot B^{\frac{3}{q}-1  }_{q, 1} } \|b \|_{ \dot B^{\frac{3}{q}+1  }_{q, 1} }.
		\end{equation}
	\item-- let $a \in \dot B^{\frac{3}{q}}_{q, 1}$ and $b \in \dot B^{\frac{3}{q}}_{q, 1},$ then 
	\begin{equation}\label{law_3}
	\|a\cdot\nabla(\sigma(D)b)\|_{\dot B^{\frac{3}{q} }_{q, 1} }\lesssim \|a \|_{ \dot B^{\frac{3}{q}  }_{q, 1} } \|b \|_{ \dot B^{\frac{3}{q}  }_{q, 1} }.
		\end{equation}
\end{lemma}
\begin{proof}
Thanks to Lemma \ref{law1}, we know that when $1\leq q<\infty,$ $\dot{B}^{\frac{3}{q}}_{q, 1}$ is an algebra, thus by Lemma \ref{FM}, we have
\begin{align*}
\|(\sigma(D)a )\cdot\nabla b\|_{\dot B^{\frac{3}{q} }_{q, 1} }&\lesssim \|\sigma(D)a \|_{ \dot B^{\frac{3}{q} }_{q, 1} } \|\nabla b \|_{ \dot B^{\frac{3}{q}  }_{q, 1} }\\
&\lesssim \|a \|_{ \dot B^{\frac{3}{q}-1  }_{q, 1} } \|b \|_{ \dot B^{\frac{3}{q}+1  }_{q, 1} }.
\end{align*}
Similarly, 
\begin{align*}
\|a\cdot\nabla(\sigma(D)b )\|_{\dot B^{\frac{3}{q} }_{q, 1} }&\lesssim \|a\|_{ \dot B^{\frac{3}{q} }_{q, 1} } \|\nabla(\sigma(D)b ) \|_{ \dot B^{\frac{3}{q}  }_{q, 1} }\\
&\lesssim \|a \|_{ \dot B^{\frac{3}{q}  }_{q, 1} } \|\sigma(D)\nabla b \|_{ \dot B^{\frac{3}{q}  }_{q, 1} }\\
&\lesssim \|a \|_{ \dot B^{\frac{3}{q}  }_{q, 1} } \|b\|_{ \dot B^{\frac{3}{q}  }_{q, 1} }.
\end{align*}
\end{proof}
The basic heat equation reads:
 \begin{equation}\label{5638295629}
 \left\{
\begin{split}
&\partial_{t}u-\mu\Delta u = f &\hbox{in } ~\mathbb{R}_+\times\mathbb{R}^3, \\
&u|_{t=0} = u_{0}  &\hbox{in }  ~\mathbb{R}^3.
\end{split}
\right. 
\end{equation} 
Then, it is classical that for all 
$u_{0}\in \mathcal{S}'(\mathbb R^{d})$ and $f\in L^1_{loc}(\mathbb R_{+}; \mathcal{S}'(\mathbb R^{d})),$ the heat equation \eqref{5638295629} has a unique tempered distribution solution, 
 which is given by  the following Duhamel's formula:
\begin{equation}
u(t)=e^{\mu t\Delta}u_{0}+\int_{0}^{t} e^{(t-s)\mu\Delta}f(\tau)~ds,\qquad t\geq0.\label{2.1}
\end{equation}
The following fundamental results to heat semi-group has been first proved in  \cite{MR1145160}.
\begin{lemma}\label{563865}
	Let $s>0$, $1 \leq p <\infty.$
		Assume that $u_{0}\in\dot B^{s}_{p, 1},$  then for any $\rho\in[1,\infty),$ 
		\begin{equation*}
		\lim_{T\to0}\|e^{t\mu\Delta}u_{0}\|_{{L}_T^\rho(\dot B^{s+\frac{2}{\rho}}_{p, 1})}=0. 
		\end{equation*}
\end{lemma}
\begin{lemma}\label{Le_27}
	Let $T>0$, $s\in\mathbb{R}$ and $1\leq p \leq\infty$. Assume that $u_{0}\in\dot B^{s}_{p, 1}$ and $f\in{L}^{1}_{T}(\dot B^{s}_{p, 1})$. Then \eqref{5638295629} has a unique solution $u$ in $\mathcal{C}([0,T]; \dot B^{s}_{p, 1})\cap{L}^{1}(0, T; \dot B^{s+2}_{p, 1})$ and there exists a constant $C$ such that 
	\begin{equation}\label{2.2}
	\|u\|_{{L}^\infty_TL\dot B^{s}_{p, 1}}+\mu\|u\|_{{L}^{1}_{T}(\dot B^{s+2}_{p, 1})}\leq C\left(\|u_{0}\|_{\dot B^{s}_{p, 1}} + \|f\|_{{L}^{1}_{T}(\dot B^{s}_{p, 1})}\right).
	\end{equation}
\end{lemma}

\section{The proof of Theorem \ref{th}}\label{se}

In this section, we shall  give the proof of Theorem \ref{th}. We only look at the case $p<q,$ since the case $p=q$ is been shown in \cite{2019arXiv191103246D}.

By means of the ${\bf{Leray~projector}}$  $\mathcal{P}:=\rm{Id}-\nabla(-\Delta)^{-1}{\rm{div}}$ and the fact that $u,\,b$ are divergence free vector fields, we can rewrite the system \eqref{main1} as (see \cite{2019arXiv191103246D}):
\begin{equation}\label{eq44}
\left\{
\begin{aligned}
&\partial_t u-\mu\Delta u=Q_{a}(b, b)-Q_{a}(u, u),\\
&\partial_t b-\nu\Delta b=Q_{b}(u-  \eps J, b),\\
&\partial_t J-\nu\Delta J=\nabla\times Q_{b}(u- \eps J, {\rm{curl}^{-1}}J),\\
\end{aligned}
\right.
\end{equation}
 and supplemented with divergence free initial data
\begin{equation}\label{ini44}
 (\mathnormal u(0,\mathnormal x), \mathnormal b(0,\mathnormal x), \mathnormal J(0,\mathnormal x))=(\mathnormal u_{0}, \mathnormal b_{0}, J_{0}).
\end{equation}
Where bi-linear forms
\begin{equation*}
\begin{split}
Q_{a}(v, w){:=}&\,\frac{1}{2}\mathcal{P}({\rm{div}}(v\otimes w)+{\rm{div}}(w\otimes v)),    \\
Q_{b}(v, w){:=}&\,{\rm{div}}(v\otimes w)-{\rm{div}}(w\otimes v),
\end{split}
\end{equation*}
and 
$$\Bigl(\div (v\otimes w)\Bigr)^j:=\sum_{k=1}^3\partial_{k}(v^jw^k).$$

Define free solution
\begin{equation*}
u_L:=e^{\mu t\Delta} u_0, ~  b_L:=e^{\nu t\Delta}b_0,  ~ J_L:=e^{\nu t\Delta} J_0.  
\end{equation*}
Let $u_0\in\dot{B}^{\frac{3}{p}-1}_{p, 1}$ and $b_0,\,J_0\in\dot{B}^{\frac{3}{q}-1}_{q, 1},$ it is easy to find  that by Lemma \ref{Le_27} that
\begin{equation*}
\begin{split}
 (u_L, b_L, J_L)\in E_p\times E_q\times E_q
 \end{split}
\end{equation*}
and there holds
\begin{equation}\label{1234}
\begin{split}
 \|u_L\|_{{L}^{\infty}(\dot B^{\frac{3}{p}-1}_{p, 1}) } + \mu\|u_L\|_{  {L}^{1}(\dot B^{\frac{3}{p}+1}_{p, 1})} \leq C   \|u_0\|_{ \dot B^{\frac{3}{p}-1}_{p, 1}  },
\end{split}
\end{equation}
\begin{equation}\label{1235}
\begin{split}
\|b_L\|_{{L}^{\infty}(\dot B^{\frac{3}{q}-1}_{q, 1}) } + \nu\|b_L\|_{  {L}^{1}(\dot B^{\frac{3}{q}+1}_{q, 1})} \leq C   \|b_0\|_{ \dot B^{\frac{3}{q}-1}_{q, 1}  },
\end{split}
\end{equation}
\begin{equation}\label{1236}
\begin{split}
\|J_L\|_{{L}^{\infty}(\dot B^{\frac{3}{q}-1}_{q, 1}) } + \nu\|J_L\|_{ {L}^{1}(\dot B^{\frac{3}{q}+1}_{q, 1})} \leq C   \|J_0\|_{ \dot B^{\frac{3}{q}-1}_{q, 1}  }.
\end{split}
\end{equation}

Define  $(\bar u, \bar b,\bar J) := (u-u^{L}, b-b^{L}, J-J^{L})$. Then $(u, b, J)$ is a solution of \eqref{eq44} if and only if $(\bar u, \bar b, \bar J)$ satisfies the following system:
 \begin{equation}\label{eq4}
\left\{
\begin{aligned}
& \partial_t\bar u-\mu\Delta\bar u=Q_{a}(\bar b, \bar b)+Q_{a}(b_L, \bar b)+Q_{a}(\bar b, b_L)+Q_{a}(b_L, b_L)\\
&\hspace*{3cm}-Q_{a}(\bar u, \bar u)-Q_{a}(u_L, \bar u)-Q_{a}(\bar u, u_L)-Q_{a}(u_L, u_L),\\
& \partial_t\bar b-\nu\Delta\bar b=Q_{b}(\bar u-\eps\bar J, \bar b)+Q_{b}(u_L-\eps J_L, \bar b)+Q_{b}(\bar u-\eps\bar J, b_L)\\
&\hspace*{7.5cm}+Q_{b}(u_L-\eps J_L, b_L),\\
& \partial_t\bar J-\nu\Delta\bar J=
\nabla\times\Bigl(Q_{b}(\bar u-\eps\bar J, {\rm{curl}^{-1}}\bar J)+Q_{b}(u_L-\eps J_L, {\rm{curl}^{-1}}\bar J)\\
&\hspace*{3cm}+Q_{b}(\bar u-\eps\bar J, {\rm{curl}^{-1}}J_L)
+Q_{b}(u_L-\eps J_L, {\rm{curl}^{-1}}J_L).
\end{aligned}
\right.
\end{equation}

In what follows, we will employing the iterative method to prove that there exists an unique solution $(\bar u, \bar b, \bar J)$ to the system \eqref{eq4}.  More precisely, the iterating approximate system is constructed as follows:
 \begin{equation}\label{iter}
 \left\{
 \begin{aligned}
 &\partial_t\bar u_n-\mu\Delta\bar u_n=Q_{a}(\bar b_{n-1}, \bar b_{n-1})+Q_{a}(b_L, \bar b_{n-1})+Q_{a}(\bar b_{n-1}, b_L)+Q_{a}(b_L, b_L)\\
&\hspace*{2.25cm}-Q_{a}(\bar u_{n-1}, \bar u_{n-1})-Q_{a}(u_L, \bar u_{n-1})-Q_{a}(\bar u_{n-1}, u_L)-Q_{a}(u_L, u_L),\\
&\partial_t\bar b_n-\mu\Delta\bar b_n=Q_{b}(\bar u_{n-1}-\eps\bar J_{n-1}, \bar b_{n-1})+Q_{b}(u_L-\eps J_L, \bar b_{n-1})\\
&\hspace*{3cm}+Q_{b}(\bar u_{n-1}-\eps\bar J_{n-1}, b_L)+Q_{b}(u_L-\eps J_L, b_L),\\
& \partial_t\bar J_n-\mu\Delta\bar J_n
 =  Q_{b}(\bar u_{n-1}-\eps\bar J_{n-1}, {\rm{curl}^{-1}}\bar J_{n-1})+Q_{b}(u_L-\eps J_L, {\rm{curl}^{-1}}\bar J_{n-1})\\
&\hspace*{3cm}+Q_{b}(\bar u_{n-1}-\eps\bar J_{n-1}, {\rm{curl}^{-1}}J_L)
+Q_{b}(u_L-\eps J_L, {\rm{curl}^{-1}}J_L).
 \end{aligned}
 \right.
 \end{equation}
We start the approximate system with 
\begin{equation*}
(\bar u_0, \bar b_0, \bar J_0)(t, x)=(0, 0, 0)
\end{equation*}
for all $t\geq 0,$
and assume the initial data of the iterative approximate system \eqref{iter} satisfied for all $n\in \mathbb{N}$
\begin{equation}\label{iteri}
 (\bar u_n, \bar b_n, \bar J_n)(0, x)=(0, 0, 0). 
\end{equation}
In the arguments proving the convergence $(n \rightarrow \infty)$ of the iterative approximate solutions of  \eqref{iter}-\eqref{iteri}, it is essential to obtain uniform estimates for it.
\subsection{Uniform boundedness of $(\bar{u}_n, \bar{b}_n, \bar{J}_n)$}\label{5632659257}
 We claim that there exists a positive constant $M$ such that for all $n \in \mathbb{N},$
\begin{equation}\label{767677767}
\begin{split}
&  \left(\|\bar u_n\|_{{L}^{\infty}(\dot B^{\frac{3}{p}-1}_{p, 1}) } + \mu\|\bar u_n\|_{ {L}^{1}(\dot B^{\frac{3}{p}+1}_{p, 1})}  \right)  \\
 &\quad\quad\quad+  \left(\| \bar b_n \|_{{L}^{\infty}(\dot B^{\frac{3}{q}-1}_{q, 1}) } + \nu\| \bar b_n \|_{ {L}^{1}(\dot B^{\frac{3}{q}+1}_{q, 1})}  \right) \\
 &\qquad\quad\quad\quad\quad+   \left(\|  \bar J_n \|_{{L}^{\infty}(\dot B^{\frac{3}{q}-1}_{q, 1}) } + \nu\|  \bar J_n \|_{ {L}^{1}(\dot B^{\frac{3}{q}+1}_{q, 1})}  \right)  \leq M.
\end{split}
\end{equation}

Obviously,  \eqref{767677767} is satisfied when $n=0$. Assume the claim \eqref{767677767} holds true for $n-1$, i.e., 
\begin{equation*}
\begin{split}
&  \|\bar u_{n-1}\|_{{L}^{\infty}(\dot B^{\frac{3}{p}-1}_{p, 1}) } + \mu\|\bar u_{n-1}\|_{  {L}^{1}(\dot B^{\frac{3}{p}+1}_{p, 1})}   
+  \| \bar b_{n-1} \|_{{L}^{\infty}(\dot B^{\frac{3}{q}-1}_{q, 1}) }\\
&\hspace*{1cm} + \nu\| \bar b_{n-1} \|_{  {L}^{1}(\dot B^{\frac{3}{q}+1}_{q, 1})}+   \|  \bar J_{n-1} \|_{{L}^{\infty}(\dot B^{\frac{3}{q}-1}_{q, 1}) } + \nu\|  \bar J_{n-1} \|_{  {L}^{1}(\dot B^{\frac{3}{q}+1}_{q, 1})}    \leq M.
\end{split}
\end{equation*}
 With smallness {condition} \eqref{small}, we now devote to the proof of \eqref{767677767}  through finding some suitable $M.$ Firstly, we need to state the following product laws for quadratic terms $Q_a, Q_b$, it will play a significant  role in the later parts.

If \eqref{condition1} is assumed, by using Lemma \ref{law1}, Lemma \ref{law2},
direct calculation  tells us that:
\begin{equation}\label{law00}
\begin{split}
 \|Q_a(v, w)\|_{\dot B^{\frac{3}{p}-1}_{p, 1} } 
 \lesssim& \|v\otimes w\|_{\dot B^{\frac{3}{p}}_{p, 1} } \\
  \lesssim& \|v\|_{\dot B^{\frac{3}{p}}_{p, 1} }\|w\|_{\dot B^{\frac{3}{p}}_{p, 1} },
  \end{split}
\end{equation}
and
\begin{equation}\label{law11}
\begin{split}
\|Q_b(v, w)\|_{\dot B^{\frac{3}{q}-1}_{q, 1} } 
 \lesssim& \|v\otimes w\|_{\dot B^{\frac{3}{q}}_{q, 1} } \\
  \lesssim& \|v\|_{\dot B^{\frac{3}{q}}_{q, 1} }\|w\|_{\dot B^{\frac{3}{q}}_{q, 1} },
\end{split}
\end{equation}
and
\begin{equation}\label{law22}
\begin{split}
\|\nabla\times Q_b(v, {\rm{curl^{-1}}}w)\|_{\dot B^{\frac{3}{q}-1}_{q, 1} } 
 \lesssim& \|Q_b(v, {\rm{curl^{-1}}}w)\|_{\dot B^{\frac{3}{q}}_{q, 1} } \\
  \lesssim& \|v\cdot\nabla ({\rm{curl^{-1}}}w)\|_{\dot B^{\frac{3}{q}}_{q, 1} }+\|({\rm{curl^{-1}}}w)\cdot\nabla v\|_{\dot B^{\frac{3}{q}}_{q, 1} }\\
  \lesssim&\|v\|_{\dot B^{\frac{3}{q}}_{q, 1}}\|w\|_{\dot B^{\frac{3}{q}}_{q, 1}}+\|v\|_{\dot B^{\frac{3}{q}+1}_{q, 1}}\|w\|_{\dot B^{\frac{3}{q}-1}_{q, 1}},
\end{split}
\end{equation}
and
\begin{equation}\label{law33}
\begin{split}
\|\nabla\times Q_b(v, {\rm{curl^{-1}}}w)\|_{\dot B^{\frac{3}{q}-1}_{q, 1} } 
 \lesssim& \|Q_b(v, {\rm{curl^{-1}}}w)\|_{\dot B^{\frac{3}{q}}_{q, 1} } \\
  \lesssim& \|v\cdot\nabla ({\rm{curl^{-1}}}w)\|_{\dot B^{\frac{3}{q}}_{q, 1} }+\|({\rm{curl^{-1}}}w)\cdot\nabla v\|_{\dot B^{\frac{3}{q}}_{q, 1} }\\
  \lesssim&\|v\|_{\dot B^{\frac{3}{p}}_{p, 1}}\|w\|_{\dot B^{\frac{3}{q}}_{q, 1}}+\|v\|_{\dot B^{\frac{3}{p}+1}_{p, 1}}\|w\|_{\dot B^{\frac{3}{q}-1}_{q, 1}}.
\end{split}
\end{equation}
Thanks to \eqref{law_0} in Lemma \ref{law0} with $\theta =\frac{3}{p}-\frac{3}{q}\leq1,$ we have
\begin{equation}\label{law44}
\begin{split}
 \|Q_a(v, w)\|_{\dot B^{\frac{3}{p}-1}_{p, 1} } 
 \lesssim& \|v\otimes w\|_{\dot B^{\frac{3}{p}}_{p, 1} } \\
  \lesssim& \|v\|_{\dot B^{\frac{6}{q}-\frac{3}{p}}_{q, 1} }\|w\|_{\dot B^{\frac{3}{p}}_{q, 1} }+\|w\|_{\dot B^{\frac{6}{q}-\frac{3}{p}}_{q, 1} }\|v\|_{\dot B^{\frac{3}{p}}_{q, 1} }.
  \end{split}
\end{equation}
Notice that  from Lemma \ref{Le_27},  there exists a constant $C$ such that
\begin{equation}\label{111}
\|u_L\|_{ {L}^{1}(\dot B^{\frac{3}{p}+1 }_{p, 1})}+ \|u_L\|_{{L}^{2}(\dot B^{\frac{3}{p} }_{p, 1}) }\leq C\eps_0(1+\frac{1}{\mu}),
\end{equation}
\begin{equation}\label{112}
\begin{split}
&\| (b_L, J_L) \|_{ {L}^{1}(\dot B^{\frac{3}{q}+1 }_{q, 1}) }+\| (b_L, J_L) \|_{ {L}^{2}(\dot B^{\frac{3}{q} }_{q, 1}) } \\
&\qquad+\|b_L  \|_{ {L}^{\frac{2pq}{pq-3q+3p}}\dot B^{\frac{6}{q}-\frac{3}{p}} _{q, 1}) }+ \|b_L \|_{{L}^{\frac{2pq}{pq+3q-3p}}(\dot B^{\frac{3}{p}}_{q, 1})}\leq C\eps_0(1+\frac{1}{\nu}).
\end{split}
\end{equation}

 Combining \eqref{law00}, \eqref{law44}, take use of interpolation inequality in Proposition \ref{P_Besov} and H\"older inequality, remember the "norm" of free solution  $(u_L, b_L, J_L)$ is small thanks to \eqref{111} and \eqref{112}, it follows from Lemma \ref{Le_27} that
\begin{equation*}
\begin{split}
 &  \|\bar u_n\|_{{L}^{\infty}(\dot B^{\frac{3}{p}-1}_{p, 1}) } + \mu\|\bar u_n\|_{  {L}^{1}(\dot B^{\frac{3}{p}+1}_{p, 1})}    \\
 \leq & C\|Q_{a}(\bar b_{n-1}, \bar b_{n-1})+Q_{a}(b_L, \bar b_{n-1})+Q_{a}(\bar b_{n-1}, b_L)+Q_{a}(b_L, b_L)\|_{{L}^{1}(\dot B^{\frac{3}{p}-1}_{p, 1})}\\
 &\quad+C\|Q_{a}(\bar u_{n-1}, \bar u_{n-1})+Q_{a}(u_L, \bar u_{n-1})+Q_{a}(\bar u_{n-1}, u_L)+Q_{a}(u_L, u_L)\|_{L^{1}(\dot B^{\frac{3}{p}-1}_{p, 1})}\\
\leq& C\int_0^\infty\Bigl(\|\bar b_{n-1}\|_{\dot B^{\frac{6}{q}-\frac{3}{p}}_{q, 1} }\|\bar b_{n-1}\|_{\dot B^{\frac{3}{p}}_{q, 1} }+\| b_L\|_{\dot B^{\frac{6}{q}-\frac{3}{p}}_{q, 1} }\|\bar b_{n-1}\|_{\dot B^{\frac{3}{p}}_{q, 1}}++\| b_L\|_{\dot B^{\frac{3}{p}}_{q, 1} }\|\bar b_{n-1}\|_{\dot B^{\frac6q-\frac{3}{p}}_{q, 1}}\\
&\quad
+\|b_L\|_{\dot B^{\frac{6}{q} -\frac3p}_{q, 1} }\|b_L\|_{\dot B^{\frac3p}_{q, 1} } +\|\bar u_{n-1}\|_{\dot B^{ \frac{3}{p}}_{p, 1} }\|\bar u_{n-1}\|_{\dot B^{\frac{3}{p}}_{p, 1} }+2\| u_L\|_{\dot B^{ \frac{3}{p}}_{p, 1} }\|\bar u_{n-1}\|_{\dot B^{\frac{3}{p}}_{p, 1}}\\
&\quad 
+\|u_L\|_{\dot B^{ \frac{3}{p}}_{p, 1} }\|u_L\|_{\dot B^{\frac{3}{p}}_{p, 1} } \Bigr)(t)\,dt\\
\leq& C\Bigl(M^2+ \eps_0 M+\eps_0^2\Bigr)\Bigl(1+\frac{1}{\nu}+\frac{1}{\mu}\Bigr)^2,
\end{split}
\end{equation*}
where we have used the facts 
\begin{equation*}
\begin{split}
 &\int_0^\infty\| b_L\|_{\dot B^{\frac{6}{q}-\frac{3}{p}}_{q, 1} }\|\bar b_{n-1}\|_{\dot B^{\frac{3}{p}}_{q, 1}}(t)\mathrm{d}t \\
 \leq &C\| b_L\|_{L^{\frac{2pq}{pq-3q+3p}}(\dot B^{\frac{6}{q}-\frac{3}{p}}_{q, 1} )} \|\bar b_{n-1}\|_{L^{\frac{2pq}{pq+3q-3p}}(\dot B^{\frac{3}{p}}_{q, 1})},
\end{split}
\end{equation*}
and
\begin{equation*}
\begin{split}
&\int_0^\infty\|\bar b_{n-1}\|_{\dot B^{\frac{6}{q}-\frac{3}{p}}_{q, 1} }\|b_L\|_{\dot B^{\frac{3}{p}}_{q, 1} }(t)\mathrm{d}t \\
\leq &C \|\bar b_{n-1}\|_{L^{\frac{2pq}{pq-3q+3p}}(\dot B^{\frac{3}{q}+1}_{q, 1})}\| b_L\|_{L^{\frac{2pq}{pq+3q-3p}}(\dot B^{ \frac{3}{p}}_{q, 1} )}.
\end{split}
\end{equation*}
 Again, taking advantage of \eqref{law00}, \eqref{law11}, interpolation and H\"older inequality, it follows from Lemma \ref{Le_27} that
\begin{equation*}
\begin{split}
 &  \|\bar b_n\|_{{L}^{\infty}(\dot B^{\frac{3}{q}-1}_{q, 1}) } + \nu\|\bar b_n\|_{  {L}^{1}(\dot B^{\frac{3}{q}+1}_{q, 1})}    \\
 \leq & C\|Q_{b}(\bar u_{n-1}-\eps\bar J_{n-1}, \bar b_{n-1})+Q_{b}(u_L-\eps J_L, \bar b_{n-1})\|_{{L}^{1}(\dot B^{\frac{3}{q}-1}_{q, 1})}\\
 &\quad+C\| Q_{b}(\bar u_{n-1}-\eps\bar J_{n-1}, b_L) +Q_{b}(u_L-\eps J_L, b_L)\|_{L^{1}(\dot B^{\frac{3}{q}-1}_{q, 1})}\\
 \leq&C \int_0^\infty\Bigl(    \|\bar b_{n-1}\|_{\dot B^{ \frac{3}{q}}_{q, 1} }\|\bar u_{n-1}\|_{\dot B^{\frac{3}{p}}_{p, 1} }+ \|\bar b_{n-1}\|_{\dot B^{ \frac{3}{q}}_{q, 1} }\|\bar J_{n-1}\|_{\dot B^{\frac{3}{q}}_{q, 1} }+\| u_L\|_{\dot B^{ \frac{3}{p}}_{p, 1} }\|\bar b_{n-1}\|_{\dot B^{\frac{3}{q}}_{q, 1}}\\
 &\quad+\|\bar b_{n-1}\|_{\dot B^{ \frac{3}{q}}_{q, 1} }\|J_L\|_{\dot B^{\frac{3}{q}}_{q, 1} } +\|\bar u_{n-1}\|_{\dot B^{ \frac{3}{p}}_{p, 1} }\|b_L\|_{\dot B^{\frac{3}{q}}_{q, 1} } +\|\bar J_{n-1}\|_{\dot B^{ \frac{3}{q}}_{q, 1} }\|b_L\|_{\dot B^{\frac{3}{q}}_{q, 1} }  \\
 &\quad+\|  u_L\|_{\dot B^{ \frac{3}{p}}_{p, 1} }\|b_L\|_{\dot B^{\frac{3}{q}}_{q, 1} }   +\|  J_L\|_{\dot B^{ \frac{3}{q}}_{q, 1} }\|b_L\|_{\dot B^{\frac{3}{q}}_{q, 1} }  \Bigr)(t)\mathrm{d}t\\
\leq& C\Bigl(M^2+ \eps_0 M+\eps_0^2\Bigr)\Bigl(1+\frac{1}{\nu}+\frac{1}{\mu}\Bigr)^2,
\end{split}
\end{equation*}
and
\begin{equation*}
\begin{split}
 &  \|\bar J_n\|_{{L}^{\infty}(\dot B^{\frac{3}{p}-1}_{p, 1}) } + \nu\|\bar J_n\|_{  {L}^{1}(\dot B^{\frac{3}{p}+1}_{p, 1})}    \\
 \leq & C\|Q_{b}(\bar u_{n-1}-\eps\bar J_{n-1}, {\rm{curl}^{-1}}\bar J_{n-1})+Q_{b}(u_L-\eps J_L, {\rm{curl}^{-1}}\bar J_{n-1}) \|_{{L}^{1}(\dot B^{\frac{3}{q}}_{q, 1})}\\
 &\quad+C\|  Q_{b}(\bar u_{n-1}-\eps\bar J_{n-1}, {\rm{curl}^{-1}}J_L)
 +Q_{b}(u_L-\eps J_L, {\rm{curl}^{-1}}J_L)\|_{L^{1}(\dot B^{\frac{3}{q}}_{q, 1})}\\
 \leq&C \int_0^\infty\Bigl(    \|\bar J_{n-1}\|_{\dot B^{ \frac{3}{q}}_{q, 1} }\|\bar J_{n-1}\|_{\dot B^{\frac{3}{q}}_{q, 1} }+ \|\bar u_{n-1}\|_{\dot B^{ \frac{3}{p}}_{p, 1} }\|\bar J_{n-1}\|_{\dot B^{\frac{3}{q}}_{q, 1} }+ \|\bar J_{n-1}\|_{\dot B^{ \frac{3}{q}+1}_{q, 1} }\|\bar J_{n-1}\|_{\dot B^{\frac{3}{q}-1}_{q, 1} }\\
 &\qquad+ \|\bar u_{n-1}\|_{\dot B^{ \frac{3}{p}+1}_{p, 1} }\|\bar J_{n-1}\|_{\dot B^{\frac{3}{q}-1}_{q, 1} }+\| u_L\|_{\dot B^{ \frac{3}{p}}_{p, 1} }\|\bar J_{n-1}\|_{\dot B^{\frac{3}{q}}_{q, 1}}+\|u_L\|_{\dot B^{ \frac{3}{p}}_{p, 1} }\|J_L\|_{\dot B^{\frac{3}{q}}_{q, 1} }\\
 &\qquad +\|\bar J_{n-1}\|_{\dot B^{ \frac{3}{q}-1}_{q, 1} }\|u_L\|_{\dot B^{\frac{3}{p}+1}_{p, 1} } +\|\bar J_{n-1}\|_{\dot B^{ \frac{3}{q}-1}_{q, 1} }\|J_L\|_{\dot B^{\frac{3}{q}+1}_{q, 1} } +\|  \bar u_{n-1}\|_{\dot B^{ \frac{3}{p}}_{p, 1} }\|J_L\|_{\dot B^{\frac{3}{q}}_{q, 1} } \\
 &\qquad+\|\bar J_{n-1}\|_{\dot B^{ \frac{3}{q} }_{q, 1} }\|J_L\|_{\dot B^{\frac{3}{q} }_{q, 1} }  +\|\bar J_{n-1}\|_{\dot B^{ \frac{3}{q}+1}_{q, 1} }\|J_L\|_{\dot B^{\frac{3}{q}-1}_{q, 1} }+\|\bar u_{n-1}\|_{\dot B^{ \frac{3}{p}+1}_{p, 1} }\|J_L\|_{\dot B^{\frac{3}{q}-1}_{q, 1} } \\
 &\qquad  +\|  J_L\|_{\dot B^{ \frac{3}{q}}_{q, 1} }\|J_L\|_{\dot B^{\frac{3}{q}}_{q, 1} } +\|  J_L\|_{\dot B^{ \frac{3}{q}}_{q, 1} }\|u_L\|_{\dot B^{\frac{3}{p}}_{p, 1} }+\|  J_L\|_{\dot B^{ \frac{3}{q}-1}_{q, 1} }\|J_L\|_{\dot B^{\frac{3}{q}+1}_{q, 1} } \\
 &\qquad +\|  u_L\|_{\dot B^{ \frac{3}{p}+1}_{p, 1} }\|J_L\|_{\dot B^{\frac{3}{q}-1}_{q, 1} }     \Bigr)(t)\mathrm{d}t\\
 \leq& C\Bigl(M^2+ \eps_0 M+\eps_0^2\Bigr)\Bigl(1+\frac{1}{\nu}+\frac{1}{\mu}\Bigr)^2.
\end{split}
\end{equation*}
By choosing $\eps_0, M$ sufficiently small such that 
\begin{equation}\label{Tsmall}
\begin{split}
 &M \leq \frac{ 1}{9C\left( 1+\frac{1}{\mu}+\frac{  1}{\nu}\right)^2 }, \qquad \varepsilon_0  \leq \frac{ 1}{9C\left(  1+\frac{1}{\mu}+\frac{  1}{\nu}\right)^2 },
\end{split}
\end{equation}
then one find that
\begin{equation*} 
\begin{split}
&  \|\bar u_{n }\|_{{L}^{\infty}_T(\dot B^{\frac{3}{p}-1}_{p, 1}) } + \mu\|\bar u_{n }\|_{  {L}^{1}_T(\dot B^{\frac{3}{p}+1}_{p, 1})}   \\
& \quad\quad\quad\quad+  \| \bar b_{n } \|_{{L}^{\infty}_T(\dot B^{\frac{3}{q}-1}_{q, 1}) } + \nu\| \bar b_{n } \|_{  {L}^{1}_T(\dot B^{\frac{3}{q}+1}_{q, 1})} \\
& \quad\quad\quad\quad\quad\quad\quad\quad+   \|  \bar J_{n } \|_{{L}^{\infty}_T(\dot B^{\frac{3}{q}-1}_{q, 1}) } + \nu\|  \bar J_{n } \|_{  {L}^{1}_T(\dot B^{\frac{3}{q}+1}_{q, 1})}    \leq M.
\end{split}
\end{equation*}
Arguing by induction, we conclude that \eqref{767677767}  holds true for all $n \in \mathbb{N}.$ 

Once the uniform bounds is established for $(\bar{u}_n, \bar{b}_n, \bar{J}_n)$, we shall use compactness arguments to prove convergence. 
\subsection{Convergence of $(\bar{u}_n, \bar{b}_n, \bar{J}_n )$}\label{sec2} 

We claim that $(\bar{u}_n, \bar{b}_n, \bar{J}_n )_{n\in\mathbb{N}}$ is a Cauchy sequence in $$  \Bigl({L}^{\infty}(\dot B^{\frac{3}{p}-1 }_{p, 1})\cap {L}^{1}(\dot B^{\frac{3}{p}+1 }_{p, 1})\Bigr)\times\Bigl({L}^{\infty}(\dot B^{\frac{3}{q}-1 }_{q, 1})\cap {L}^{1}(\dot B^{\frac{3}{q}+1 }_{q, 1})\Bigr)\times\Bigl({L}^{\infty}(\dot B^{\frac{3}{q}-1 }_{q, 1})\cap {L}^{1}(\dot B^{\frac{3}{q}+1 }_{q, 1})\Bigr).$$ For all $n\in\mathbb{N},$ let us consider the difference   $$(\delta\bar u^n, \delta\bar b^n, \delta\bar J^n):=(\bar{u}_{n+1}, \bar{b}_{n+1}, \bar{J}_{n+1})-(\bar{u}_{n}, \bar{b}_{n}, \bar{J}_{n} ),$$
and define
\begin{align*}
& \delta^n:=\|\delta\bar u^{n }\|_{{L}^{\infty}(\dot B^{\frac{3}{p}-1}_{p, 1}) } + \mu\|\delta\bar u^{n }\|_{  {L}^{1}(\dot B^{\frac{3}{p}+1}_{p, 1})} \\
&\quad\quad\quad+ \|\delta\bar b^{n }\|_{{L}^{\infty}(\dot B^{\frac{3}{q}-1}_{q, 1}) } + \nu\|\delta\bar u^{n }\|_{  {L}^{1}(\dot B^{\frac{3}{q}+1}_{q, 1})} \\
&\quad\quad\quad+\|\delta\bar J^{n }\|_{{L}^{\infty}(\dot B^{\frac{3}{q}-1}_{q, 1}) } + \nu\|\delta\bar J^{n }\|_{  {L}^{1}(\dot B^{\frac{3}{q}+1}_{q, 1})}.
 \end{align*}

Then thanks to Lemma \ref{Le_27}, we only need to estimate the following terms
\begin{align*}
I_1^n:=&\|Q_{a}(\delta\bar b^{n -1}, \bar b_{n})\|_{{L}^{1}(\dot B^{\frac{3}{p}-1}_{p, 1}) }+\|Q_{a}(\bar b_{n-1 }, \delta\bar b^{n -1})\|_{{L}^{1}(\dot B^{\frac{3}{p}-1}_{p, 1}) }\\
&+\|Q_{a}(b_L, \delta\bar b^{n -1})\|_{{L}^{1}(\dot B^{\frac{3}{p}-1}_{p, 1}) }+\|Q_{a}(\delta\bar b^{n -1}, b_L)\|_{{L}^{1}(\dot B^{\frac{3}{p}-1}_{p, 1}) }\\
&+\|Q_{a}(\delta\bar u^{n -1}, \bar u_{n})\|_{{L}^{1}(\dot B^{\frac{3}{p}-1}_{p, 1}) }+\|Q_{a}(\bar u_{n-1 }, \delta\bar u^{n -1})\|_{{L}^{1}(\dot B^{\frac{3}{p}-1}_{p, 1}) }\\
&+\|Q_{a}(u_L, \delta\bar u^{n -1})\|_{{L}^{1}(\dot B^{\frac{3}{p}-1}_{p, 1}) }+\|Q_{a}(\delta\bar u^{n -1}, u_L)\|_{{L}^{1}(\dot B^{\frac{3}{p}-1}_{p, 1}) },
\end{align*}

\begin{align*}
I_2^n:=&\|Q_{b}(\delta\bar u^{n -1}-\eps\delta\bar J^{n -1} , \bar b_{n})\|_{{L}^{1}(\dot B^{\frac{3}{q}-1}_{q, 1}) }+\|Q_{b}(\bar u_{n-1 }-\eps \bar J^{n -1}, \delta\bar b^{n -1})\|_{{L}^{1}(\dot B^{\frac{3}{q}-1}_{q, 1}) }\\
&+\|Q_{a}(u_L-\eps J_L, \delta\bar b^{n -1})\|_{{L}^{1}(\dot B^{\frac{3}{q}-1}_{q, 1}) }+\|Q_{a}(\delta\bar u^{n -1}-\eps\delta\bar J^{n -1}, b_L)\|_{{L}^{1}(\dot B^{\frac{3}{q}-1}_{q, 1}) },
\end{align*}

\begin{align*}
I_3^n:=&\|Q_{b}(\delta\bar u^{n -1}-\eps\delta\bar J^{n -1} , \curl^{-1}\bar J_{n})+Q_{b}(\bar u_{n-1 }-\eps \bar J^{n -1}, \curl^{-1}\delta\bar J^{n -1})\|_{{L}^{1}(\dot B^{\frac{3}{q}}_{q, 1}) }\\
&+\|Q_{a}(u_L-\eps J_L, \curl^{-1}\delta\bar J^{n -1})+Q_{a}(\delta\bar u^{n -1}-\eps\delta\bar J^{n -1}, \curl^{-1}J_L)\|_{{L}^{1}(\dot B^{\frac{3}{q}}_{q, 1}) }.
\end{align*}

Along extremely similar calculations as previous subsections,  thanks to \eqref{Tsmall} and uniform bounds \eqref{767677767}, one can get  
\begin{equation*}
\begin{split}
\delta^n&\leq C(I_1^n+I_2^n+I_3^n)\\
&\leq C\delta^{n-1}\Bigl(1+\frac1\mu+\frac1\nu\Bigr)^2\bigl(M+\eps_0\bigr)\\
&\leq \frac12 \delta^{n-1}.
\end{split}
\end{equation*}
Thus,  we know that   $(\bar{u}_n, \bar{b}_n, \bar{J}_n )_{n\in\mathbb{N}}$ is a Cauchy sequence in the space $$  \Bigl({L}^{\infty}(\dot B^{\frac{3}{p}-1 }_{p, 1})\cap {L}^{1}(\dot B^{\frac{3}{p}+1 }_{p, 1})\Bigr)\times\Bigl({L}^{\infty}(\dot B^{\frac{3}{q}-1}_{q, 1})  \cap {L}^{1}(\dot B^{\frac{3}{q}+1}_{q, 1})\Bigr) \times \Bigl({L}^{\infty}(\dot B^{\frac{3}{q}-1}_{q, 1})  \cap{L}^{1}(\dot B^{\frac{3}{q}+1}_{q, 1})\Bigr)$$
and there exists a triplet $(\bar{u}, \bar{b}, \bar{J})$ such that
\begin{equation*}
\begin{split}
\bar{u}_n \rightarrow \bar{u} \quad \text{in}& \quad{L}^{\infty}(\dot B^{\frac{3}{p}-1 }_{p, 1})\cap {L}^{1}(\dot B^{\frac{3}{p}+1 }_{p, 1}),\\
 \bar{b}_n \rightarrow \bar{b} \quad \text{in}& \quad {L}^{\infty}(\dot B^{\frac{3}{q}-1}_{q, 1})  \cap{L}^{1}(\dot B^{\frac{3}{q}+1}_{q, 1}),\\
 \bar{J}_n \rightarrow \bar{J} \quad \text{in}& \quad {L}^{\infty}(\dot B^{\frac{3}{q}-1}_{q, 1})  \cap{L}^{1}(\dot B^{\frac{3}{q}+1}_{q, 1}).
\end{split}
\end{equation*}
By the product laws that we used frequently before, it is not difficult to prove the convergence of non-linear terms in \eqref{iter}, by an example, we show that 
\begin{equation*}
\begin{split}
&\|Q_{b}(\bar u_{n}-\eps\bar J_{n}, {\rm{curl}^{-1}}\bar J_{n})-Q_{b}(\bar u-\eps\bar J, {\rm{curl}^{-1}}\bar J)\|_{{L}^{1}_T(\dot B^{\frac{3}{q}}_{q, 1})}\\
\leq&\|Q_{b}(\bar u_{n}-\bar u+\eps\bar J_{n}-\eps\bar J, {\rm{curl}^{-1}}\bar J_{n})\|_{{L}^{1}(\dot B^{\frac{3}{q}}_{q, 1})}\\
&\quad\quad+\|Q_{b}(\bar u-\eps\bar J, {\rm{curl}^{-1}}\bar J-{\rm{curl}^{-1}}\bar J_n)\|_{{L}^{1}(\dot B^{\frac{3}{q}}_{q, 1})}\\
\lesssim&\, M\Bigl(\|\bar u_{n}-\bar u\|_{{L}^{2}_T(\dot B^{\frac{3}{p}}_{p, 1})\cap {L}^{1}(\dot B^{\frac{3}{p}+1}_{p, 1})}+\|\eps\bar J_{n}-\eps\bar J\|_{{L}^{\infty}(\dot B^{\frac{3}{q}-1}_{q, 1})\cap {L}^{1}(\dot B^{\frac{3}{q}+1}_{q, 1})}\Bigr).
\end{split}
\end{equation*}

Hence,  we conclude that $(\bar{u} , \bar{b}, \bar{J})$ is indeed a solution of \eqref{eq44}. This implies that  $({u} ,{b}, J  )= (u_L+\bar{u}, b_L+\bar{b}, J_L+\bar{J})$ is a solution of \eqref{main1} in  $$  \Bigl({L}^{\infty}(\dot B^{\frac{3}{p}-1 }_{p, 1})\cap {L}^{1}(\dot B^{\frac{3}{p}+1 }_{p, 1})\Bigr)\times\Bigl({L}^{\infty}(\dot B^{\frac{3}{q}-1}_{q, 1})  \cap {L}^{1}(\dot B^{\frac{3}{q}+1}_{q, 1})\Bigr) \times \Bigl({L}^{\infty}(\dot B^{\frac{3}{q}-1}_{q, 1})  \cap{L}^{1}(\dot B^{\frac{3}{q}+1}_{q, 1})\Bigr)$$
and satisfies \eqref{1.1200}.

The continuity of $(u, B, J)$ is straightforward. Indeed, the right-hand sides  of \eqref{eq44} belong to $L^1_T(\dot{B}_{p, 1}^{\frac3p-1}),$ $L^1_T(\dot{B}_{q, 1}^{\frac3q-1}),$ $ L^1_T(\dot{B}_{q, 1}^{\frac3q-1})$ respectively.

\subsection{Uniqueness}
Let $(u_1, b_1, J_1)$ and $(u_2, b_2, J_2)$ be two solutions of \eqref{eq44} in $$E_p\times E_q\times E_q$$ with same initial data $(u_0, b_0, J_0)\in \dot{B}^{\frac{3}{p}-1}_{p, 1}\times \dot{B}^{\frac{3}{q}-1}_{q, 1}\times\dot{B}^{\frac{3}{q}-1}_{q, 1}.$ Without loss of generality, we assume  $(u_2, b_2, J_2)$ is the solution that constructed in the previous steps (in fact, one only needs $J_2$ is small in $E_q).$

Set $\delta u:= u_2-u_1$, $\delta b:= b_2-b_1$, and $\delta J:= J_2-J_1$,  we see that $(\delta u, \delta b, \delta J )$ satisfies
\begin{equation*}
\left\{
\begin{aligned}
&\partial_t \delta u -\mu\Delta  \delta u =Q_{a}( \delta b , b_1)+Q_{a}(  b_2 , \delta b  )-Q_{a}(\delta u , u_1 )-Q_{a}(u_2 , \delta u ),\\
&\partial_t  \delta b-\nu\Delta  \delta b =Q_{b}(u_1-\eps J_1, \delta b)+Q_{b}(\delta u-\eps\delta J, b_2),\\
&\partial_t  \delta J-\nu\Delta  \delta J =\nabla\times Q_{b}(u_1-\eps J_1, {\rm{curl}^{-1}}\delta J)+\nabla\times Q_{b}(\delta u-\eps\delta J, {\rm{curl}^{-1}} J_2),\\
&(  \delta u(0,  x),  \delta b(0,  x),   \delta J(0,  x))=(  0,  0, 0).
\end{aligned}
\right.
\end{equation*}

With our assumptions on two solutions,  one can verify that the right-hand sides of above system belong to $L^1(\dot{B}^{\frac{3}{p}-1}_{p, 1}),$ $L^1(\dot{B}^{\frac{3}{q}-1}_{q, 1}),$ $L^1(\dot{B}^{\frac{3}{q}-1}_{q, 1})$ respectively, thus by Lemma \ref{Le_27}, Lemma \ref{law0} (take $\theta=\frac3p-\frac3q$) and product laws \eqref{law00}-\eqref{law33}, one has
\begin{equation*}
\begin{split}
&\|\delta u (t)\|_{\dot B^{\frac3p-1}_{p,1}}+ \|(\delta b, \delta J)(t)\|_{\dot B^{\frac3q-1}_{q,1}}
+\int_0^t\Bigl(\|\delta u (t)\|_{\dot B^{\frac3q+1}_{p,1}}+ \|(\delta b, \delta J)(t)\|_{\dot B^{\frac3q+1}_{q,1}}\Bigr)\,d\tau\\
&\lesssim \int_0^t\Bigl(\|(b_1, b_2) \|_{\dot B^{\frac 3p}_{q,1}}\|\delta b\|_{\dot B^{\frac 6q-\frac3p}_{q,1}} +\|(b_1, b_2) \|_{\dot B^{\frac6q-\frac3p}_{q,1}}\|\delta b\|_{\dot B^{\frac3p}_{q,1}}+\|(u_1, u_2) \|_{\dot B^{\frac 3p}_{p,1}}\|\delta u\|_{\dot B^{\frac3p}_{p,1}}\\
&\quad\quad\quad+(\|u_1 \|_{\dot B^{\frac 3p}_{p,1}}+\|J_1\|_{\dot B^{\frac3q}_{q,1}})\|\delta b\|_{\dot B^{\frac 3q}_{q,1}}+\|b_2\|_{\dot B^{\frac3q}_{q,1}}(\|\delta u\|_{\dot B^{\frac3p}_{p,1}}+\|\delta J\|_{\dot B^{\frac3q-1}_{q,1}})\\
&\quad\quad\quad+(\|u_1 \|_{\dot B^{\frac 3p}_{p,1}}+\|(J_1, J_2) \|_{\dot B^{\frac 3q}_{q,1}})\|\delta J \|_{\dot B^{\frac 3q}_{q,1}}+(\|u_1 \|_{\dot B^{\frac 3p+1}_{p,1}}+\|(J_1, J_2) \|_{\dot B^{\frac 3q+1}_{q,1}})\|\delta J \|_{\dot B^{\frac 3q-1}_{q,1}}\\
&\quad\quad\quad+\|\delta u\|_{\dot B^{\frac 3p}_{p,1}}\|J_2 \|_{\dot B^{\frac 3q}_{q,1}}+\|\delta u\|_{\dot B^{\frac 3p+1}_{p,1}}\|J_2 \|_{\dot B^{\frac 3q-1}_{q,1}}
\Bigr)\mathrm{d}\tau,\\
&\lesssim \int_0^t\Omega(\tau)\bigl(\|\delta u (\tau)\|_{\dot B^{\frac3p-1}_{p,1}}+ \|(\delta b, \delta J)(\tau)\|_{\dot B^{\frac3q-1}_{q,1}}\bigr)\mathrm{d}\tau,
\end{split}
\end{equation*}
where 
\begin{align*}
&\Omega(\tau):=\Bigl(\|(u_1, u_2)(\tau)\|_{\dot B^{\frac3p-1}_{p,1}}+1\Bigr)\|(u_1, u_2)(\tau)\|_{\dot B^{\frac3p+1}_{p,1}}\\
&\hspace{3cm}+\Bigl(\|(b_1,b_2, J_1,J_2)(\tau)\|_{{\dot B^{\frac3q-1}_{q,1}}}+1\Bigr)\|(b_1,b_2, J_1,J_2)(\tau)\|_{{\dot B^{\frac3q+1}_{q,1}}}.
\end{align*}
It is clear that our assumptions ensure $\Omega\in L^1(\R_+),$ Gronwall {lemma}  then enables  us to conclude that $(\delta u, \delta b, \delta J)\equiv 0$ on $\R_+\times \R^3.$

For completing the proof of the   existence for the original Hall-MHD system, we  have to check that  $J_0=\nabla\times b_0$ implies $J=\nabla\times b$,
so that $(u, b)$ is indeed a distributional solution of \eqref{1.1}--\eqref{1.3}. 
Actually, it is easy to see that $\nabla\times b_L=J_L,$ thus we only need to check $\bar J=\nabla\times\bar b.$ Noticing that
\begin{align*}
&(\partial_t-\Delta)(\nabla\times \bar b-\bar J)=\nabla\times  Q_b(u-\eps J, {\rm{curl}}^{-1}(\nabla\times \bar b-\bar J)).
\end{align*}
Hence, using Lemma \ref{Le_27} and product law  \eqref{law11}, one gets for all $t\geq 0,$ 
\begin{equation*}
\begin{split}
&\|(\nabla\times \bar b-\bar J)(t)\|_{\dot B^{\frac3q-2}_{q,1}}
+\int_0^t\|\nabla\times \bar b-\bar J\|_{\dot B^{\frac3q}_{q,1}}\,d\tau\\
\leq &C\int_0^t\|u-\eps J \|_{\dot B^{\frac 3q}_{q,1}}\|\nabla\times \bar b- \bar J\|_{\dot B^{\frac 3q-1}_{q,1}} \mathrm{d}\tau.
\end{split}
\end{equation*}
Then  \eqref{562875} combined with interpolation inequality and Gronwall {lemma}  enables  that 
$\nabla\times \bar b- \bar J\equiv0$ on $\R_+\times\mathbb{R}^3.$

A slight modification on the proof could yields local well-posedness by assuming only  $\|\nabla\times b_0\|_{\dot B^{\frac 3q-1}_{q,1}}$  is small enough and  in addition
$$-\frac1{3}<\frac1q-\frac1q.$$
In fact, by Lemma \ref{563865} one can guarantee \eqref{111} and \eqref{112} are
satisfied with sufficient small time $T>0.$ And Lemma \ref{Le_27} implies that
$$\|J_L\|_{L^\infty_T(\dot{B}^{\frac3q-1}_{q, 1})}\lesssim\|\nabla\times b_0\|_{\dot{B}^{\frac3q-1}_{q, 1}}\lesssim \eps_0.$$
We  omit another details here and thus the proof of Theorem \ref{th} is completed.\quad$\square$


\section{The proof of Theorem \ref{th1}}
In this section, we devote to proving the decay estimates of the solution provided by Theorem \ref{th}. Started with the data $(u_0, b_0)$ satisfies \eqref{small},  we know that the solution $(u, b)\in E_p\times E_q$ such that $\nabla\times b\in E_q$ and satisfies \eqref{1.1200}. 

{For  any fixed $m\geq 1 ,$ let $T\geq0$ be the largest $t$ such that}
\begin{equation*}
W(t):=\sup_{0\leq \tau\leq t} \tau^{\frac{ m }{2}}\left( \|{D^m} u(\tau)\|_{\dot{B}_{p,1}^{\frac3p-1}}+\|{D^m} b(\tau)\|_{\dot{B}_{q,1}^{\frac3q-1}}\right)\leq C_0\eps_0,
\end{equation*}
where $C_0$ will be chosen later.
\subsection{Decay estimates for velocity fields}
Applying $\dot{\Delta}_j$ to the equation \eqref{1.1} and    taking  ${D^\alpha_x}$ ($|\alpha|=m$) on the resulting equation leads to
\begin{equation*}
\begin{split}
\partial_t \dot{\Delta}_j{D^\alpha_x}u-\mu\Delta\dot{\Delta}_j{D^\alpha_x} u=\dot{\Delta}_j{D^\alpha_x}\mathcal{P}\div(b\otimes b)-\dot{\Delta}_j{D^\alpha_x} \mathcal{P}\div(u\otimes u).
\end{split}
\end{equation*}
Then 
\begin{equation*}
\dot{\Delta}_j{D^\alpha_x}u=e^{t\mu\Delta}{\Delta}_j{D^\alpha_x}u_0+\int_0^t e^{(t-s)\mu\Delta}\mathcal{P}\left( \dot{\Delta}_j{D^\alpha_x} \div(b\otimes b)-\dot{\Delta}_j{D^\alpha_x} \div(u\otimes u) \right)\,ds.
\end{equation*}
Lemma \ref{semi}  thus implies that 
\begin{equation}\label{uu} 
\begin{split}
\| \dot{\Delta}_j{D^\alpha_x} u\|_{L^p}  \leq&   Ce^{- c\mu2^{2j}t} \| \dot{\Delta}_j{D^\alpha_x} u_0\|_{L^p}\\
&+C\int_0^t  e^{-c\mu2^{2j}(t-s)}   \left( \|\dot{\Delta}_j{D^\alpha_x} \mathcal{P}\div(b\otimes b)\|_{L^p}\right.\\
&\qquad\qquad\qquad\left.+\|\dot{\Delta}_j{D^\alpha_x} \mathcal{P}\div(u\otimes u)\|_{L^p} \right)\mathrm{d}s  \\
\leq& C   e^{-c\mu2^{2j}t} \| \dot{\Delta}_j{D^\alpha_x} u_0\|_{L^p}+A_1+A_2+A_3,
\end{split}
\end{equation} 
where
\begin{equation*} 
\begin{split}
A_1 := &C\int_0^{\frac t2}  e^{- c\mu2^{2j}(t-s)}  \left( \|\dot{\Delta}_j  {D^\alpha_x} \mathcal{P}\div(b\otimes b)\|_{L^p}+\|\dot{\Delta}_j {D^\alpha_x}  \mathcal{P}\div(u\otimes u)\|_{L^p} \right)\mathrm{d}s,\qquad
\end{split}
\end{equation*} 
\begin{equation*} 
\begin{split}
A_2:= &C\int_{\frac t2}^t   e^{- c\mu2^{2j}(t-s)} 2^{  j }     \|\dot{\Delta}_j D^{\alpha-1}_x \mathcal{P}\div(b\otimes b)\|_{L^p}\mathrm{d}s,
\end{split}
\end{equation*} 
\begin{equation*} 
\begin{split}
A_3:= &C\int_{\frac t2}^t   e^{-  c\mu2^{2j}(t-s)} 2^{  j }    \|\dot{\Delta}_jD^{\alpha-1}_x \div(u\otimes u)\|_{L^p} \mathrm{d}s.
\end{split}
\end{equation*} 
Noticing there exists a {constant} $\wt c>0$ such that
\begin{equation}\label{eexp}
e^{- c\mu2^{2j}t}2^{jk} \leq  e^{-\wt c\mu2^{2j}t} t^{-\frac{k}{2}}, \qquad \text{for~any}~ k\geq 0.
\end{equation}
{By} employing Proposition \ref{563856}, a straightforward calculation shows that
\begin{equation}\label{A1} 
\begin{split}
A_1 \leq &C t ^{-\frac{ m }{2}}  \int_0^{\frac t2}  e^{-\wt c\mu2^{2j}(t-s)}  \left( \|\dot{\Delta}_j  \div(b\otimes b)\|_{L^p}+\|\dot{\Delta}_j  \div(u\otimes u)\|_{L^p} \right)\mathrm{d}s\\
\leq&C g_j t ^{-\frac{ m }{2}}  2^{-\left(\frac3p-1\right)j}     \Bigl( \|   \div(b\otimes b)\|_{L^1(\dot{B}_{p, 1}^{\frac3p-1})}+\|  \div(u\otimes u)\|_{L^1(\dot{B}_{p, 1}^{\frac3p-1})}\Bigr) \\
\leq&C g_jt ^{-\frac{ m }{2}}  2^{-\left(\frac3p-1\right)j}  \left( \|  u \|_{E_p}^2+ \|  b \|_{ E_q}^2\right)\\
\leq&C g_jt ^{-\frac{ m }{2}}  2^{-\left(\frac3p-1\right)j} (1+\frac1\mu+\frac1\nu)\eps_0^2,
\end{split}
\end{equation}
{where} $\{g_j\}_{j\in\mathbb{Z}}\in \ell^1$ and $\|\{g_j\}\|_{\ell^1}\leq 1.$

\noindent Thanks to \eqref{law_000} , we have
\begin{equation*} \label{de.11}
\begin{split}
\|\dot{\Delta}_j{D^{\alpha-1}_x} \mathcal{P}\div(b\otimes b)\|_{L^p}\leq&C2^j \|\dot{\Delta}_j  (D^{\alpha-1}_xb\otimes b )\|_{L^p}\\
\leq&Cg_j2^{-\left( \frac3q-1\right)j} \|   D^{\alpha-1}_xb\otimes b \|_{\dot{B}_{p, 1}^{\frac3q } }\\
\leq&Cg_j2^{-\left( \frac3q-1\right)j} \|    D^{\alpha-1}_xb  \|_{\dot{B}_{q, 1}^{\frac6q-\frac3p } } \|  b \|_{\dot{B}_{q, 1}^{\frac3q} }\\
&\quad+Cg_j2^{-\left( \frac3q-1\right)j} \|    D^{\alpha-1}_xb  \|_{\dot{B}_{q, 1}^{\frac3q } } \|  b \|_{\dot{B}_{q, 1}^{\frac6q-\frac3p } }.
\end{split}
\end{equation*}
By means of interpolation , we get
\begin{equation*} \label{De.11}
\begin{split}
&\|    D^{\alpha-1}_x b  \|_{\dot{B}_{q, 1}^{\frac6q-\frac3p } } \|  b \|_{\dot{B}_{q, 1}^{\frac3q } }+\|    D^{\alpha-1}_x b  \|_{\dot{B}_{q, 1}^{\frac3q } } \|  b \|_{\dot{B}_{q, 1}^{\frac6q-\frac3p } }\\
\lesssim& \Bigl(\|    D^{\alpha}_x b  \|_{\dot{B}_{q, 1}^{\frac3q-1 } }^{1-r} \|  b \|_{\dot{B}_{q, 1}^{\frac3q-1 } } ^{r}\Bigr)\Bigl(\|    D^{\alpha}_x b  \|_{\dot{B}_{q, 1}^{\frac3q-1 } }^{\frac1m} \|  b \|_{\dot{B}_{q, 1}^{\frac3q-1 } } ^{1-\frac1m}\Bigr)\\
&\quad+ \Bigl(\|    D^{\alpha}_x b  \|_{\dot{B}_{q, 1}^{\frac3q-1 } }\Bigr) \Bigl(\|    D^{\alpha}_x b  \|_{\dot{B}_{q, 1}^{\frac3q-1 } }^{\frac1m-r} \|  b \|_{\dot{B}_{q, 1}^{\frac3q-1 } } ^{1-\frac1m+r}\Bigr)\\
\lesssim&  \|    D^{\alpha}_x b  \|_{\dot{B}_{q, 1}^{\frac3q-1 } }^{1+\frac1m-r} \|  b \|_{\dot{B}_{q, 1}^{\frac3q-1 } } ^{1-\frac1m+r},
\end{split}
\end{equation*}
where $r:=\frac3m({\frac1p-\frac1q})\leq\frac1m.$

\noindent Since $\frac3p-\frac3q<1,$  \eqref{eexp} and H\"older inequality {imply} that
\begin{equation}\label{A2} 
\begin{split}
A_2\leq&  Cg_j2^{-\left( \frac3p-1\right)j} \int_{\frac{t}2}^t (t-s)^{-\frac12(1+\frac3p-\frac3q)} \| D^{\alpha}_x b  \|_{\dot{B}_{q, 1}^{\frac3q-1 } }^{1+\frac1m-r} \|  b \|_{\dot{B}_{q, 1}^{\frac3q-1 } } ^{1-\frac1m+r}\,ds\, \\
\leq& C g_j 2^{-\left( \frac3p-1\right)j} (\frac{t}2)^{-\frac{m}2(1+\frac1m-r)}W^{1+m-r}\eps_0^{1-\frac1m+r} \int_{\frac{t}2}^t (t-s)^{-\frac12(1+\frac3p-\frac3q)} \,ds\\
\leq& 2^{\frac{m}2+1}C g_j 2^{-\left( \frac3p-1\right)j} t^{-\frac{m}2}W^{1+\frac1m-r}\eps_0^{1-\frac1m+r}.
\end{split}
\end{equation} 
Thanks to \eqref{law00}, we have
\begin{equation*} \label{de.22}
\begin{split}
\|\dot{\Delta}_j{D^{\alpha-1}_x} \mathcal{P}\div(u\otimes u)\|_{L^p} 
\leq&C2^j \|\dot{\Delta}_j  (D^{\alpha-1}_xu\otimes u )\|_{L^\infty(L^p)}\\
\leq&Cg_j2^{-\left( \frac3p-1\right)j}  \|    D^{\alpha-1}_xu\otimes u  \|_{\dot{B}_{p, 1}^{\frac3p } }\\
\leq&Cg_j2^{-\left( \frac3p-1\right)j}  \|    D^{\alpha-1}_xu  \|_{\dot{B}_{p, 1}^{\frac3p } } \|  u \|_{\dot{B}_{p, 1}^{\frac3p } }.
\end{split}
\end{equation*} 
By means of interpolation,
\begin{equation*} \label{De.22}
\begin{split}
&\|    D^{\alpha-1}_x u  \|_{\dot{B}_{p, 1}^{\frac3p } } \|  u \|_{\dot{B}_{p, 1}^{\frac3p } }\\
\lesssim& \|    D^{\alpha}_x u  \|_{\dot{B}_{p, 1}^{\frac3p-1 } }\|    D^{\alpha}_x u  \|_{\dot{B}_{p, 1}^{\frac3p-1 } }^{\frac1m} \|  u \|_{\dot{B}_{p, 1}^{\frac3p-1 } } ^{1-\frac1m}.
\end{split}
\end{equation*}
\noindent Thus, \eqref{eexp} and  H\"older inequality  {imply} that
\begin{equation}\label{A3} 
\begin{split}
A_3\leq&  Cg_j2^{-\left( \frac3p-1\right)j}\, \int_{\frac{t}2}^t (t-s)^{-\frac12}\|    D^{\alpha}_x u  \|_{\dot{B}_{p, 1}^{\frac3p-1 } }^{1+\frac1m} \|  u \|_{\dot{B}_{p, 1}^{\frac3p-1 } } ^{1-\frac1m}\,ds \\
\leq& C g_j 2^{-\left( \frac3p-1\right)j} (\frac{t}2)^{-\frac{m+1}2}W^{1+\frac1m}\eps_0^{1-\frac1m}\int_{\frac{t}2}^t (t-s)^{-\frac12}\,ds.\\
\leq& 2^{\frac{m}2+1}C g_j 2^{-\left( \frac3p-1\right)j} t^{-\frac{m}2}W^{1+\frac1m}\eps_0^{1-\frac1m}.
\end{split}
\end{equation} 
Putting \eqref{A1}, \eqref{A2} and \eqref{A3} together, one has
\begin{equation} \label{Du}
\begin{split}
 t^{\frac{m}2}  \|D^m u\|_{\dot{B}^{\frac3p-1}_{p, 1}}\lesssim \varepsilon_0+ (1+\frac1\mu+\frac1\nu){\eps_0}^2+W^{1+\frac1m-r}{\eps_0}^{1-\frac1m+r}+W^{1+\frac1m}{\eps_0}^{1-\frac1m}.
\end{split}
\end{equation}
\subsection{Decay estimates for Magnetic fields}
Applying $\dot{\Delta}_j$ to equation \eqref{1.3}  and    taking  ${D^\alpha}$ on the resulting equation leads to
\begin{equation}\label{598670}
\begin{split}
\partial_t \dot{\Delta}_j {D^{ \alpha}} b-\nu\Delta\dot{\Delta}_j {D^\alpha} b=&\dot{\Delta}_j {D^\alpha} \nabla\times((u-  \eps J)\times b).
\end{split}
\end{equation}
Then 
\begin{equation*}
\dot{\Delta}_j{D^\alpha_x}b=e^{t\nu\Delta}{\Delta}_j{D^\alpha_x}b_0+\int_0^t e^{(t-s)\nu\Delta}\left( \dot{\Delta}_j{D^\alpha_x} \nabla\times(u\times b)-\eps\dot{\Delta}_j{D^\alpha_x} \nabla\times(J\times b) \right)\,ds.
\end{equation*}
Lemma \ref{semi}  thus implies that 
\begin{equation}\label{bb} 
\begin{split}
\| \dot{\Delta}_j{D^\alpha_x} b\|_{L^q}  \leq&   Ce^{- c\nu2^{2j}t} \| \dot{\Delta}_j{D^\alpha_x} b_0\|_{L^q}\\
&+C\int_0^t  e^{-c\nu2^{2j}(t-s)}   \left( \|\dot{\Delta}_j{D^\alpha_x} \nabla\times(u\times b)\|_{L^p}\right.\\
&\left.\qquad\qquad\qquad-\|\dot{\Delta}_j{D^\alpha_x} \nabla\times(\eps J\times b)\|_{L^q} \right)\mathrm{d}s  \\
\leq& C   e^{-c\nu2^{2j}t} \| \dot{\Delta}_j{D^\alpha_x} b_0\|_{L^q}+A_4+A_5+A_6,
\end{split}
\end{equation} 
where
\begin{equation*} 
\begin{split}
A_4 := &C\int_0^{\frac t2}  e^{- c\nu2^{2j}(t-s)}  \left( \|\dot{\Delta}_j  {D^\alpha_x} \nabla\times(u\times b)\|_{L^q}+\|\dot{\Delta}_j {D^\alpha_x}  \nabla\times(\eps J\times b)\|_{L^q} \right)\mathrm{d}s,\qquad
\end{split}
\end{equation*} 
\begin{equation*} 
\begin{split}
A_5:= &C\int_{\frac t2}^t   e^{- c\nu2^{2j}(t-s)} 2^{  j }     \|\dot{\Delta}_j D^{\alpha-1}_x \nabla\times(u\times b)\|_{L^q}\mathrm{d}s,
\end{split}
\end{equation*} 
\begin{equation*} 
\begin{split}
A_6:= &C\int_{\frac t2}^t   e^{-  c\nu2^{2j}(t-s)} 2^{  j }   \|\dot{\Delta}_jD^{\alpha-1}_x \nabla\times( \eps J\times b)\|_{L^q} \mathrm{d}s.
\end{split}
\end{equation*} 
Similar with the method of getting estimates  \eqref{A1} and \eqref{A3}, one can easily  show that
\begin{equation}\label{A4} 
\begin{split}
A_4 \leq &C t ^{-\frac{ m }{2}}  \int_0^{\frac t2}  e^{-\wt c\mu2^{2j}(t-s)}  \left( \|\dot{\Delta}_j \nabla\times(u\times b)\|_{L^q}+\|\dot{\Delta}_j  \nabla\times(\eps J\times b)\|_{L^q} \right)\mathrm{d}s\\
\leq&C g_j t ^{-\frac{ m }{2}}  2^{-\left(\frac3q-1\right)j}     \Bigl( \|   \nabla\times(u\times b)\|_{L^1(\dot{B}_{q, 1}^{\frac3q-1})}+\|  \nabla\times(\eps J\times b))\|_{L^1(\dot{B}_{q, 1}^{\frac3q-1})}\Bigr) \\
\leq&C g_jt ^{-\frac{ m }{2}}  2^{-\left(\frac3q-1\right)j}  \left( \|  u \|_{E_p}^2+ \|  (b , J)\|_{ E_q }^2\right)\\
\leq&C g_jt ^{-\frac{ m }{2}}  2^{-\left(\frac3q-1\right)j} (1+\frac1\mu+\frac1\nu)\eps_0^2,
\end{split}
\end{equation}
and
\begin{equation}\label{A5} 
\begin{split}
A_5\leq&  Cg_j2^{-\left( \frac3q-1\right)j} \int_{\frac{t}2}^t (t-s)^{-\frac12}\|D^{\alpha-1}_x (u\times b)  \|_{\dot{B}_{q, 1}^{\frac3q } }\,ds\\
\leq&  Cg_j2^{-\left( \frac3q-1\right)j} \int_{\frac{t}2}^t (t-s)^{-\frac12}\Bigl(\|D^{\alpha-1}_x u\times b  \|_{\dot{B}_{q, 1}^{\frac3q } }+\|u\times D^{\alpha-1}_x b  \|_{\dot{B}_{q, 1}^{\frac3q } }\Bigr)\,ds\\
\leq&  Cg_j2^{-\left( \frac3q-1\right)j}  W^{1+\frac1m}\eps_0^{1-\frac1m}\int_{\frac{t}2}^t (t-s)^{-\frac12}s^{-\frac12(1+m)}\,ds\\
\leq&  2^{\frac{m}2+1}Cg_j2^{-\left( \frac3q-1\right)j} t^{-\frac{m}2}W^{1+\frac1m}\eps_0^{1-\frac1m}.
\end{split}
\end{equation} 
\noindent Because $\div b=0,$ one can rewrite
$$\nabla\times(\eps J\times b)=\eps\nabla\times\bigl(\div(b\otimes b)-\nabla(\frac{|b|^2}{2})\bigr)=\eps\nabla\times\bigl(\div(b\otimes b)\bigr),$$
then, H\"older inequality yields
\begin{equation}\label{A6}
\begin{split}
A_6\leq& C\eps\int_{\frac t2}^t   e^{-  c\nu2^{2j}(t-s)} 2^3{  j }   \|\dot{\Delta}_jD^{\alpha-1}_x b\otimes b\|_{L^q} \mathrm{d}s\\
\leq& C\eps2^{-(\frac3q-1)j}g_j\int_{\frac t2}^t   e^{-  c\nu2^{2j}(t-s)} 2^{ 2 j }   \|D^{\alpha-1}_x (b\otimes b)\|_{\dot{B}_{q, 1}^{\frac3q }} \mathrm{d}s\\
\leq& C\eps2^{-(\frac3q-1)j}g_j\int_{\frac t2}^t   e^{-  c\nu2^{2j}(t-s)} 2^{ 2 j }   \|D^{\alpha}_x b\|_{\dot{B}_{q, 1}^{\frac3q-1 }}\| \nabla\times b\|_{\dot{B}_{q, 1}^{\frac3q-1 }} \mathrm{d}s\\
\leq& C\eps\eps_02^{-(\frac3q-1)j}g_j(\frac{t}2)^{-\frac{m}2}\int_{\frac t2}^t   e^{-  c\nu2^{2j}(t-s)} 2^{ 2 j }   s^\frac{m}{2}\|D^{\alpha}_x b\|_{\dot{B}_{q, 1}^{\frac3q-1 }}\mathrm{d}s\\
\leq& 2^{\frac{m}{2}+1}C\frac\eps\nu2^{-(\frac3q-1)j}g_jt^{-\frac{m}2}W(T)\eps_0,
\end{split}
\end{equation}
here we use {the fact that}
$$\int_{\frac t2}^t   e^{-  c\nu2^{2j}(t-s)} 2^{ 2 j } \,ds\leq\frac{2}{c\nu}.$$

Putting \eqref{A4}, \eqref{A5} and \eqref{A6} together, one has
\begin{equation*} \label{Db}
\begin{split}
 t^{\frac{m}2}  \|D^m b\|_{\dot{B}^{\frac3q-1}_{q, 1}}\lesssim\varepsilon_0+ (1+\frac1\mu+\frac1\nu){\eps_0}^2+W^{1+\frac1m}{\eps_0}^{1-\frac1m}+{\frac\eps\nu}W\eps_0.
\end{split}
\end{equation*}
This combine with \eqref{Du} implies that
\begin{equation*} 
\begin{split}
W(T) 
&\leq C\varepsilon_0+C\left( 1+\frac1\mu+\frac1\nu+\frac\eps\nu C_0+C_0^{{1+\frac1m-r}}+C_0^{{1+\frac1m}}\right)\eps_0^2.
\end{split}
\end{equation*}
Thanks to \eqref{Tsmall}, one can take suitable $C_0$ such that  $W(T) <\frac12C_0\eps_0.$ By the continuous induction,  we have $W(t)\leq C_0\eps_0$ for all $t\geq 0.$ It completes the proof of Theorem \ref{th1}.\quad$\square$


\section{The proof of Theorem \ref{th2}}\label{se5}
In order to prove the theorem ( $\mu=\nu$ is assumed), we need to notice from \cite{2019arXiv191103246D}  that if  $(u, b)$ is a  solution of Hall-MHD system \eqref{1.1}-\eqref{1.3} in the sense of distribution, then the so-called \emph{velocity of electron $v:=u-\eps J$} satisfies:
\begin{equation}\label{ve}
\begin{aligned}
&\partial_tv-\mu\Delta v=\mathcal{P}\bigl(\div(b\otimes b)-\div (u\otimes u)\bigr)-\eps\nabla\times((\nabla\times v)\times b)\\
&\hspace{4cm}+\nabla\times(v\times u)+2\eps\nabla\times(v\cdot\nabla b).
\end{aligned}
\end{equation}
The equation \eqref{ve} is still quasi-linear compare to the equation of current $J$. However, owing to \begin{equation*}((\nabla\times(\nabla\times v))\times b, v)_{L^2}=((\nabla\times v)\times b, \nabla\times v)_{L^2}=0,\end{equation*}
the most nonlinear term cancels out when performing an energy method. Thus, contrasting to the uniqueness part of Theorem \ref{th}, it will help us to release the smallness assumption on current $J$.

We now focusing to the proof of Theorem \ref{th2}. Since  $(u_{0,2}, b_{0,2})$ satisfies the initial conditions in  \cite{2019arXiv191103246D} Theorem 2.2 about the local well-posedness of Hall-MHD.  Thus, supplemented with initial data $(u_{0,2}, b_{0,2})$ there exists a solution $(u_2, b_2)$ on the maximal time  interval $[0,T^*)$
fulfilling 
$$(u_2, b_2, \nabla\times b_2)\in E_2(t),$$
for all $t<T^*.$

Define $v_i:=u_i-\eps J_i,~(i=1,2).$ It is then convenient to consider the difference $(\wt u, \wt b, \wt v):=(u_1-u_2, b_1-b_2, v_1-v_2),$ which satisfies:

\begin{equation}\label{d}
\left\{\begin{aligned}
 &\partial_t {\wt u}-\mu\Delta {\wt u}:=d_1,\\
 &\partial_t {\wt b}-\mu\Delta {\wt b}:=d_2,\\
 &\partial_t \wt v-\mu\Delta\wt v:=d_1+d_3+d_4+ d_5,\\
 &(\wt u, \wt B, \wt v)|_{t=0}{=}(u_{0,1}-u_{0, 2}, b_{0,1}-b_{0, 2},v_{0,1}-v_{0, 2}),
\end{aligned}
\right.
\end{equation}
where
\begin{align*}
&d_1:=\mathcal{P}\bigl(-\div(\wt b\otimes \wt b)+\div(\wt b\otimes b_1)+\div(b_1\otimes \wt b)+\div(\wt u\otimes \wt u)-\div(\wt u\otimes u_1)\\&\quad\quad\quad-\div(u_1\otimes \wt u)\bigr),\\
&d_2:=\nabla\times(-\wt v\times\wt b+v_1\times\wt b+\wt v\times b_1),\\
&d_3:=-\eps\nabla\times(-(\nabla\times \wt v)\times\wt b+(\nabla\times v_1)\times\wt b+(\nabla\times \wt v)\times b_1),\\
&d_4:=\eps\nabla\times(-\wt v\times\wt u+ v_1\times\wt u+\wt v\times u_1),\\
&d_5:=2\eps\nabla\times(-\wt v\cdot\nabla\wt b+v_1\cdot\nabla\wt b+\wt v\cdot\nabla b_1).
\end{align*}
We know that $(\wt u, \wt b, \wt v)\in  L^\infty_t(\dot{B}^\frac12_{2, 1})\cap L^1_t(\dot{B}^\frac52_{2, 1})$ since both $(u_i, b_i, v_i)$ belong to that space. Now, we shall estimate the difference $(\wt u, \wt b, \wt v)$ in the space $\dot{B}^\frac12_{2, 1},$ one thus has to verify $d_1$ to $d_5$ live in the space  $L^1_T(\dot{B}^\frac12_{2, 1})$ firstly, which is quite easy.  A standard energy method gives that for all $t\in[0, T^*),$
\begin{multline}
\|(\wt u, \wt B, \wt v)(t)\|_{\dot{B}^{\frac12}_{2, 1}}+\mu\int_0^t\|(\wt u, \wt B, \wt v)\|_{\dot{B}^{\frac52}_{2, 1}}\,d\tau\lesssim\int_0^t\Bigl(\|(d_1, d_2, d_4, d_5)\|_{\dot{B}^{\frac12}_{2, 1}}\\
+\|\nabla\times((\nabla\times v_1)\times\wt b)\|_{\dot{B}^{\frac12}_{2, 1}}+\sum_{j\in\mathbb{Z}}2^{\frac{3j}{2}}\|[\ddj, (b_1+\wt b)\times](\nabla\times\wt v)\|_{L^2}
\,\Bigr)d\tau\label{5.555}.
\end{multline}
Using the fact that $\dot{B}^\frac12_{2, 1}$ is an algebra and 
$$\|b\|_{\dot{B}^\frac32_{2,1}}\sim\|(u, v)\|_{\dot{B}^\frac12_{2,1}}, \quad\|\nabla b\|_{\dot{B}^\frac32_{2,1}}\sim\|(u, v)\|_{\dot{B}^\frac32_{2,1}}.$$
one has 
$$\begin{aligned}
\|d_1\|_{\dot{B}^{\frac12}_{2, 1}}&\lesssim\|\wt b\|_{\dot{B}^{\frac32}_{2, 1}}^2+\|\wt b\|_{\dot{B}^{\frac32}_{2, 1}}\|b_1\|_{\dot{B}^{\frac32}_{2, 1}}+\|\wt u\|_{\dot{B}^{\frac32}_{2, 1}}^2+\|\wt u\|_{\dot{B}^{\frac32}_{2, 1}}\|u_1\|_{\dot{B}^{\frac32}_{2, 1}},\\
\|d_2\|_{\dot{B}^{\frac12}_{2, 1}}&\lesssim\|\wt v\|_{\dot{B}^{\frac32}_{2, 1}}\|\wt b\|_{\dot{B}^{\frac32}_{2, 1}}+\|v_1\|_{\dot{B}^{\frac32}_{2, 1}}\|\wt b\|_{\dot{B}^{\frac32}_{2, 1}}+\|\wt v\|_{\dot{B}^{\frac32}_{2, 1}}\|b_1\|_{\dot{B}^{\frac32}_{2, 1}},\\
\|d_4\|_{\dot{B}^{\frac12}_{2, 1}}&\lesssim\|\wt v\|_{\dot{B}^{\frac32}_{2, 1}}\|\wt u\|_{\dot{B}^{\frac32}_{2, 1}}+\|v_1\|_{\dot{B}^{\frac32}_{2, 1}}\|\wt u\|_{\dot{B}^{\frac32}_{2, 1}}+\|\wt v\|_{\dot{B}^{\frac32}_{2, 1}}\|u_1\|_{\dot{B}^{\frac32}_{2, 1}},\\
\|d_5\|_{\dot{B}^{\frac12}_{2, 1}}
&\lesssim \|\wt v\|_{\dot{B}^{\frac32}_{2, 1}}\|\nabla\wt b\|_{\dot{B}^{\frac32}_{2, 1}}+\|\wt v\|_{\dot{B}^{\frac32}_{2, 1}}\|\nabla b_1\|_{\dot{B}^{\frac32}_{2, 1}}+\|v_1\|_{\dot{B}^{\frac32}_{2, 1}}\|\nabla\wt b\|_{\dot{B}^{\frac32}_{2, 1}}\\
&\lesssim \|\wt v\|_{\dot{B}^{\frac32}_{2, 1}}\|(\wt u, \wt v)\|_{\dot{B}^{\frac32}_{2, 1}}+\|(\wt u, \wt v)\|_{\dot{B}^{\frac32}_{2, 1}}\|(u_1, v_1)\|_{\dot{B}^{\frac32}_{2, 1}},
\end{aligned}$$
and
$$\|\nabla\times((\nabla\times v_1)\times\wt b)\|_{\dot{B}^{\frac12}_{2, 1}}\lesssim \|v_1\|_{\dot{B}^{\frac52}_{2, 1}}\|\wt b\|_{\dot{B}^{\frac12}_{2, 1}}.$$
Thanks to the commutator estimate (see \cite{MR3186849})
\begin{equation*}
\sum_{j\in\Z} 2^{\frac{3j}2}\|[\dot\Delta_j, w]z\|_{L^2}
\lesssim \|\nabla w\|_{\dot B^{\frac32}_{2,1}}\|z\|_{\dot B^{\frac12}_{2,1}},
\end{equation*}
we have
$$\sum_{j\in\mathbb{Z}}2^{\frac{3j}{2}}\|[\ddj, (b_1+\wt b)\times](\nabla\times\wt v)\|_{L^2}\lesssim \|(\wt u,\wt v, u_1, v_1)\|_{\dot B^{\frac32}_{2,1}}\|\wt v\|_{\dot B^{\frac32}_{2,1}}.$$

 \noindent Hence, by interpolation and Young's inequality, inequality \eqref{5.555} becomes
\begin{align}
    &\|(\wt u, \wt b, \wt v)(t)\|_{\dot{B}^{\frac12}_{2, 1}}+\mu\int_0^t\|(\wt u, \wt b, \wt v)(\tau)\|_{\dot{B}^{\frac52}_{2, 1}}\,d\tau\notag\\
\leq&\|(\wt u, \wt b, \wt v)(0)\|_{\dot{B}^{\frac12}_{2, 1}}+\int_0^t \bigl(\wt\Omega(\tau)+C\|(\wt u, \wt b,\wt v)\|_{\dot{B}^{\frac52}_{2, 1}}\bigr)\|(\wt u, \wt b, \wt v)(\tau)\|_{\dot{B}^{\frac12}_{2, 1}}\,d\tau\label{dd}
\end{align}
with $\wt\Omega(t):=C\bigl(\|(u_1, b_1, v_1)\|_{\dot{B}^{\frac32}_{2, 1}}^2+\|v_1\|_{\dot{B}^{\frac52}_{2, 1}}\bigr)\cdotp$
\medbreak
Now, one needs to prove the following bootstrap argument (see similar result in \cite{2019arXiv191209194D}).

\begin{lemma}\label{Le_5.000}
Let $X,$ $D,$ $W$ be three   nonnegative  measurable functions on $[0,T]$. Assume that there exists a  nonnegative real constant ~$C$ 
such that  for any $t\in[0,T]$, 
\begin{equation}\label{5.000}
X(t)+\mu\int_0^t D(\tau)\,d\tau\leq X(0)+\int_0^t\Bigl(\wt\Omega(\tau) X(\tau)+ CX(\tau)D(\tau)\Bigr)\,d\tau.
\end{equation}
If, in addition, 
\begin{equation}
2CX(0)\exp\biggl(\int_0^T \wt\Omega(\tau) \,d\tau\biggr) < \mu\label{5.200},
\end{equation}
then, for any $t\in[0,T]$, one has
\begin{equation}
X(t)+\frac\mu2\int_0^t D\,d\tau\leq X(0)\exp\biggl(\int_0^t \wt\Omega\,d\tau\biggr)\cdotp\label{5.300}
\end{equation}
\end{lemma}

\begin{proof}    Let $\wt T$ be the largest $t\leq T$ such that 
\begin{equation}\label{5.999} 2C\sup_{0\leq t'\leq t} X(t')\leq\mu.\end{equation}
Then, \eqref{5.000} implies that for all $t\in[0,\wt T],$ we have\begin{equation}
X(t)+\frac\mu2 \int_0^t D\,d\tau\leq X(0)+\int_0^t\wt\Omega(\tau) X(\tau)\,d\tau .\label{5.400}
\end{equation}
By Gronwall  lemma, this  yields for all $t\in[0,\wt T],$ 
$$X(t)+\frac\mu2\int_0^t D(\tau)\,d\tau\leq X(0)\exp\biggl(\int_0^t\wt\Omega(\tau)\,d\tau\biggr)\cdotp$$
Hence,  it is clear that if \eqref{5.200} is satisfied, then \eqref{5.999} is satisfied with 
a strict inequality. A continuity argument thus ensures that we must have $\wt T=T$
and thus \eqref{5.300} on $[0,T].$ 
\end{proof}
 Noticing our assumptions on $(u_1, b_1)$ ensure that $\wt\Omega \in L^1(\R_+).$ By virtue of \eqref{ind}, let $\eta$ satisfies
$$2C\eta\exp\bigl(\|\wt\Omega\|_{L^1(\R_+)}\bigr)<\mu,$$
 and apply Lemma \ref{Le_5.000} to
inequality \eqref{dd}, we have
for any $t\in[0, T^*),$
\begin{align*}
    \|(\wt u, \wt B, \wt v)(t)\|_{\dot{B}^{\frac12}_{2, 1}}+\frac\mu2\int_0^t\|(\wt u, \wt B, \wt v)(\tau)\|_{\dot{B}^{\frac52}_{2, 1}}\,d\tau\leq \eta\exp\bigl(\|\wt\Omega\|_{L^1(\R_+)}\bigr).
\end{align*}
\medbreak
The above inequality ensures that $(\wt u, \wt b, \wt v)\in L^\infty(0, T^*; \dot{B}^\frac12_{2,1})\cap L^1(0, T^*; \dot{B}^\frac52_{2,1})$ and so does $(u_2, b_2, v_2),$ thus we conclude by classic arguments that $(u_2, b_2, v_2)$  can continued beyond $T^*,$ which finally implies that $T^*=\infty.$
This completes the proof of Theorem \ref{th2}.\quad$\square$

\medskip
\textbf{Acknowledgment.} { Part of this paper was discussed when the first author visit  Universit\'e Paris-Est. The authors express
 	much gratitude to Prof. Rapha\"{e}l Danchin and Prof. Weixi Li  for their supports. The first author thank to Wuhan university's financial supports to visit Universit\'e Paris-Est. The second author is supported by the PhD fellowship from  Universit\'e Paris-Est.}

\end{document}